\documentclass[MSNbibl,number,citesort,dvips]{arxbj}
\usepackage{mathbh}

\aid{0}
\volume{18}
\issue{3}
\pubyear{2012}
\firstpage{914}
\lastpage{944}
\doi{10.3150/11-BEJ361}

\makeatletter
\newproclaim{ass}{Assumption}
\newremark{remark}{Remark}
\newtheorem{corollary}{Corollary}
\newtheorem{theorem}{Theorem}

\newtheorem{proposition}{Proposition}
\newtheorem{lemma}{Lemma}

\newcommand{\eps}{\varepsilon}
\newcommand{\opi}{\varpi}

\newcommand{\llambda}{{\bolds{\lambda}}}
\newcommand{\llambdastar}{\llambda^{*}}
\newcommand{\lambdastar}{\lambda^{*}}

\newcommand{\bomega}{\bolds{\omega}}

\newcommand{\Pb}{\mathbf{P}}
\newcommand{\bY}{\mathbf{Y}}
\newcommand{\bL}{\mathbf{L}}
\newcommand{\bZ}{\mathbf{Z}}
\newcommand{\bW}{\mathbf{W}}
\newcommand{\bX}{\mathbf{X}}
\newcommand{\RR}{\mathbb{R}}
\newcommand{\calX}{\mathcal{X}}
\newcommand{\T}{\top}
\newcommand{\1}{\mathbh{1}}
\newcommand{\tr}{\operatorname{Tr}}
\newcommand{\KL}{\mathcal{K}}
\newcommand{\sign}{\operatorname{sgn}}

\newcommand{\sff}{\mathsf{f}}
\newcommand{\Pf}{P_{f}}

\newcommand{\mcZ}{\mathcal{Z}}
\newcommand{\mcF}{\mathcal{F}}
\newcommand{\mcM}{\mathcal{M}}
\newcommand{\mcC}{\mathcal{C}}
\newcommand{\mfF}{\mathfrak{F}}
\newcommand{\mcE}{\mathcal{E}}
\newcommand{\mfL}{\mathfrak{L}}
\newcommand{\mcFL}{\mathcal{F}_{\Lambda}}
\newcommand{\mcX}{\mathcal{X}}
\newcommand{\Ex}{\mathbf{E}}
\newcommand{\mcPLambda}{\mathcal{P}_{\Lambda}}
\newcommand{\flambda}{f_{\llambda}}
\newcommand{\fllambda}{f_{\llambda}}

\newcommand{\lj}{\lambda_{j}}
\renewcommand{\hat}{\widehat}
\makeatother

\begin{document}
\begin{frontmatter}

\title{Mirror averaging with sparsity priors}
\runtitle{Mirror averaging with sparsity priors}

\begin{aug}
%%%% inicialai - be tarpu
\author[1]{\fnms{Arnak S.} \snm{Dalalyan}\corref{}\thanksref{1}\ead[label=e1]{dalalyan@imagine.enpc.fr}} \and
\author[2]{\fnms{Alexandre B.} \snm{Tsybakov}\thanksref{2}\ead[label=e2]{Alexandre.Tsybakov@ensae.fr}}
\runauthor{A.S. Dalalyan and A.B. Tsybakov}
\address[1]{LIGM/IMAGINE, Ecole des Ponts ParisTech, Universit\'{e}
Paris Est, France.\\\printead{e1}}
\address[2]{Laboratoire de Statistique, CREST and LPMA, Universit\'e
Paris 6, France.\\\printead{e2}}
\end{aug}

% HISTORY:
\received{\smonth{3} \syear{2010}}
\revised{\smonth{11} \syear{2010}}

% ABSTRACT
%
\begin{abstract}
We consider the problem of aggregating the elements of a possibly
infinite dictionary for building a decision procedure that aims at
minimizing a given criterion. Along with the dictionary, an
independent identically distributed training sample is available, on
which the performance of a~given procedure can be tested. In a
fairly general set-up, we establish an oracle inequality for the
Mirror Averaging aggregate with any prior distribution. By choosing
an appropriate prior, we apply this oracle inequality in the context
of prediction under sparsity assumption for the problems of regression
with random design, density estimation and binary classification.
\end{abstract}

% KEYWORDS
%
\begin{keyword}
\kwd{aggregation of estimators}
\kwd{mirror averaging}
\kwd{oracle inequalities}
\kwd{sparsity}
\end{keyword}

\end{frontmatter}
%

%s1 ###
%s1 #&#
\section{Introduction}\label{sec1}

In recent years, several methods of estimation and selection under the
sparsity scenario have been discussed in the literature. The
$\ell_1$-penalized least squares (Lasso) is by far the most studied one
and its statistical properties are now well understood (cf., e.g.,
\cite{BRT09,btw06,btw07a,btw07b,MeinBuhl06,vdg06,ZhangHuang08,zhang09,ZhaoYu06} and
the references cited therein). %The Lasso is particularly attractive
%by its low computational cost. For instance, one can use the LARS
%algorithm \cite{LARS}, which is quite popular.
Several other estimators are closely related to the Lasso, such as the
Elastic net \cite{HasZou05}, the Dantzig selector \cite{ct}, the
adaptive Lasso \cite{Zou06}, the least squares with entropy or
$\ell_{1+\delta}$ penalization \cite{k06,Kol08}, etc. These
estimators are obtained as solutions of convex or linear programming
problems and are attractive by their low computational cost. However,
they have good theoretical properties only under rather restrictive
assumptions, such as the mutual coherence assumption~\cite{det06}, the
uniform uncertainty/restricted isometry principle~\cite{ct}, the
irrepresentable \cite{ZhaoYu06} or the restricted eigenvalue
\cite{BRT09} conditions. Roughly speaking, these conditions mean that,
for example, in the linear regression context one should assume that
the Gram matrix of the predictors is not too far from the identity
matrix. Such type of assumption is natural if we want to identify the
parameters or to retrieve the sparsity pattern, but it is not necessary
if we are interested only in the prediction ability.

Indeed, at least in theory, there exist estimators attaining
sufficiently good accuracy of prediction under almost no assumption on
the Gram matrix. This is, in particular, the case for the
$\ell_0$-penalized least squares estimator \cite{btw04}, Theorem 3.6,
\cite{btw07a}, Theorem 3.1. However, in practice this estimator can be
unstable (cf. \cite{Breiman}).
Furthermore, its computation is an NP-hard problem, %We finally mention
%the paper \cite{Wipf},
%which brings to attention the fact that the empirical Bayes
%estimator in Gaussian regression with Gaussian prior can effectively
%recover the sparsity pattern. This method is realized in \cite{Wipf}
%via the EM alorithm. However, its theoretical properties are not
%explored, and it is not clear what are the limits of application of
%the method beyond the considered set of numerical examples.
and there is a challenge to find a method realizing a compromise
between the theoretical optimality and computational efficiency.
Motivated by this, we proposed in \cite{dt07,dt08,dt09,dt10} an
approach to estimation under the sparsity scenario, which is quite
different from the $\ell_1$ penalization techniques. The idea is to use
an exponentially weighted aggregate (EWA) with a properly chosen
sparsity-favoring prior. Let us note that there exists an extensive
literature on EWA, which does not discuss the sparsity issue. Thus,
procedures with exponential weighting are quite common in the context
of on-line learning with deterministic data, see
\cite{ccg04,HKW98,v90}, the monograph \cite{lcb06} and the
references cited therein. Statistical properties of various versions of
EWA are discussed in
\cite{Al08,a2,buno05,catbook01,Cat07,gir08,jrt08,jntv05,lb06,y01,y03,Yang04,zhang1}.

On the difference from these works, we focus in
\cite{dt07,dt08,dt09,dt10} on the ability of EWA to deal with the
sparsity issue. Specifically, we prove that EWA with a properly chosen
prior satisfies sparsity oracle inequalities (SOI), which are
comparable with those for the $\ell_0$-penalized techniques and are
even better in some aspects. At the same time, on the difference from
the $\ell_0$-penalized methods, our method is computationally feasible
for relatively large dimensions of the problem, cf. \cite{dt10}.
Furthermore, our estimator has theoretical advantages as compared to
the $\ell_1$-penalized methods, since it satisfies oracle inequalities
with leading constant 1 that hold with almost no assumption on the
dictionary/Gram matrix (cf. detailed comparison with the $\ell_1$ based
methods in Section \ref{secdisc} below).\looseness=-1

The results of \cite{dt07,dt08,dt09,dt10} are established for the
linear regression model with fixed design. The aim of this paper is to
show that similar ideas can be successfully implemented
for a
large scope of statistical problems with i.i.d. data, in particular,
for regression with random design, density estimation and
classification. The procedure that we propose is mirror averaging (MA)
with sparsity priors. The difference from the EWA considered in
\cite{dt07,dt08,dt09,dt10} is that we compute the exponential weights
recursively and then average them out.\looseness=-1

This paper is organized as follows. In Section \ref{sec2}, we introduce
some notation and formulate main assumptions. Section \ref{sec3}
contains the definition of the MA estimator and a~general PAC-Bayesian
risk bound in expectation. In Section \ref{sec4}, we introduce our
sparsity prior and obtain our main SOI as a corollary of the
PAC-Bayesian bound. Sections~\ref{sec5},~\ref{sec6} and
\ref{sec7} consider applications of this result to specific models,
namely, to nonparametric regression with random design, density
estimation and classification. In Section \ref{secdisc}, we briefly
discuss computational aspects of the MA aggregate and compare it to
other methods of sparse estimation. Technical proofs are given in the
\hyperref[appendix]{Appendix}.

%Section~\ref{sec6} presents the suggested computational algorithm
%based on Langevin Monte-Carlo (LMC), and Section~\ref{sec7} reports
%the results of numerical experiments.

%s2 ###
%s2 #&#
\section{Notation and assumptions}\label{sec2}

Let $(\mcZ,\mfF)$ be a measurable space and let $\{\Pf,f\in\mcF\}$ be a collection of probability measures on
$(\mcZ,\mfF)$ indexed by some set $\mcF$. We are interested the
estimation of $f$ based on an i.i.d.  sample $Z_1,\ldots,Z_n$ drawn
from the probability distribution $\Pf$. We will assume that $f$ is a
``functional'' parameter, that is $\mcF$ is a subset of a vector space
$\mcE=\{f\dvtx\mcX\to\RR^d\}$ for some set $\mcX$ and for some
positive integer $d$. From now on, we denote by $\Ex_f$ the expectation
w.r.t.  $\Pf$ and by $\bZ$ the random vector
$(Z_1,\ldots,Z_n)\in\mcZ^n$.

To further specify the settings, let $\ell\dvtx\mcE\times\mcF\to\RR
_+$ be a
general loss function. An estimator of $f$ is any mapping $\tilde
f\dvtx\mcZ^n\to\mcE$ such that the mapping ${\mathbf z}\mapsto\ell
(\tilde
f({\mathbf z}),f)$, defined on $(\mcZ^n,\mfF^{n})$ and with values in
$\RR_+$,
is measurable for every $f\in\mcF$. The performance of an estimator
$\tilde f$ is quantified by the risk
\[
\Ex_f[\ell(\tilde f(\bZ),f)]:=\int_{\mcZ^n} \ell(\tilde
f({\mathbf z}),f)\Pf^{n}(\mathrm{d}\mathbf{z}).
\]
Here $\Pf^{n}$ stands for the product measure $\Pf\otimes\cdots
\otimes
\Pf$ on $(\mcZ^{n},\mfF^{n})$. We will assume the following.

\setcounter{ass}{0}
\renewcommand{\theass}{Q\arabic{ass}}
\begin{ass}\label{assq1}
There exists a mapping $Q\dvtx\mcZ\times\mcE\to
\RR$
such that, for every $f\in\mcF$:
\begin{itemize}
\item[--] the mapping $z\mapsto Q(z,g)$ is measurable and
$\Pf$-integrable for every $g\in\mcE$,
\item[--] $\Delta(f)\triangleq
\int_\mcZ Q(z,g)\Pf(\mathrm{d}z)-\ell(g,f)$ is independent of $g$ and finite
for any $f\in\mcF$.
\end{itemize}
\end{ass}

Assumption \ref{assq1} is fulfilled in a number of settings; detailed discussion
is given in Sections~\ref{sec5}--\ref{sec7}. For example, in the case
of regression with squared loss, one has
$z=(x,y)\in\mcZ=\calX\times\RR$ and $\ell(g,f)=\int_{\mathcal{X}}
(g-f)^2\,\mathrm{d}P_X$, where $P_X$ stands for the distribution of the design and
$f$ is the regression function. Assumption \ref{assq1} is then fulfilled with
$Q(z,g)=(y-g(x))^2$. In simple words, Assumption \ref{assq1} requires the
existence of an unbiased estimator of the risk $\ell(g,f)$, up to a
summand depending exclusively on $f$, where $f$ is the unknown
parameter and $g$ is a known function. It is worth noting that under
Assumption \ref{assq1} the minimizer of the loss function $g\mapsto\ell(g,f)$
coincides with the minimizer of the expectation $g\mapsto\int Q(Z,g)
P_f (\mathrm{d}Z)$.
This property is crucial in what follows. %and generalizes the
%standard condition $\ell(g,f)=\int Q(Z,g) P_f (dZ)$.

Since, in general, there is no estimator having the smallest possible
risk among all possible estimators, we will pursue a more realistic
goal, which consists in finding an estimator whose risk, for every $f$,
is nearly as small as the minimal risk $\min_{g\in\mcFL} \ell(g,f)$
over a prespecified subset $\mcFL$ of $\mcE$, that is, we will follow
the oracle approach.
%This is the so
%called oracle inequalities approach, see \cite{Candes05} for a good
%review.
To make this approach sensible, the subfamily $\mcFL$ should not be too
large. On the other hand, it should be chosen large enough to contain a
good approximation to the (unknown) ``true'' function $f$.

The set $\mcFL$ is indexed by the elements of some measurable space
$(\Lambda,\mfL)$. More precisely, we define
$\mcFL=\{\flambda,\llambda\in\Lambda\}\subset\mcE$ as a
collection of functions (dictionary) such that, for every $x\in\mcX$
and $z\in\mcZ$, the mappings $\llambda\mapsto\flambda(x)$,
$\llambda\mapsto Q(z,\flambda)$ and $\llambda\mapsto\ell(\flambda,f)$
from $\Lambda$ to $\RR$ are measurable. The elements of the dictionary
$\mcFL$ can be interpreted as candidate estimators of $f$. Define
$\mcPLambda$ as the set of all probability measures on $(\Lambda,\mfL)$
and $\mathcal{P}_{1}(\mcFL)$ as the set of all measures $\mu\in\mcPLambda$ such that
$\int_\Lambda|\flambda(x)|\mu(\mathrm{d}\llambda)<\infty$ for every
$x\in\mcX$. We define for every $\mu\in\mathcal{P}_{1}(\mcFL)$,
\[
\sff_\mu=\int_\Lambda\flambda
\mu(\mathrm{d}\llambda) \biggl(\sff_\mu(x)=\int_\Lambda\flambda(x)
\mu(\mathrm{d}\llambda), \forall x\in\mcX\biggr).
\]
We say that $\sff_\mu$ is a convex aggregate of functions $\flambda$
with $\mu$ being the mixing measure or the measure of aggregation. The
estimators we study in the present work are convex aggregates with
data-dependent mixing measures.

In what follows, we denote by $\mcC(\mcFL)$ the set of all convex
aggregates of functions $\flambda$, that is
\[
\mcC(\mcFL)=\{g\dvtx\mcX\to\RR \mbox{ s.t. } g=\sff_\mu\mbox{ for some }\mu\in\mathcal{P}_{1}(\mcFL)\}.
\]
It is clear that $\mcC(\mcFL)$ is a convex set containing $\mcFL$. For
our main result, we need the following condition on the function $Q$
appearing in Assumption \ref{assq1}.

\begin{ass}\label{assq2}
There exist $\beta>0$ and a mapping
$\Psi_\beta\dvtx\mcC(\mcFL)\times\mcC(\mcFL)\to\RR_+$ such that
\begin{longlist}
\item $\Psi_\beta(g,g)=1$ for all $g\in\mcC(\mcFL)$,
\item the
mapping $g\mapsto\Psi_\beta(g,\tilde g)$ is concave on $\mcC(\mcFL)$
for every fixed $\tilde g\in\mcC(\mcFL)$,
\item the inequality
\[
\int_\mcZ\exp\bigl(-\beta^{-1}\{Q(z,g)-Q(z,\tilde
g)\}\bigr)\Pf(\mathrm{d}z)\le\Psi_\beta(g,\tilde g)
\]
holds for every $g,\tilde g\in\mcC(\mcFL)$.
\end{longlist}
\end{ass}

At first sight, this assumption seems cumbersome but we will show that
it holds for a~number of settings which are of central interest in
nonparametric statistics. For example, in the model of regression with
random design and additive Gaussian noise, Assumption~\ref{assq2} is fulfilled
for $\beta\ge2\sigma^2+2\sup_{\lambda}\|f_\llambda-f\|^2_\infty$,
where $\sigma^2$ is the noise variance and~$f$ is the unknown
regression function. Assumption~\ref{assq2} has been first introduced
in \cite{jrt08}, Theorem~4.2 for finite dictionaries and a variant of
it has been used in~\cite{a2}, Corollary~5.1.

Note also that\vspace*{-2pt} if Assumption \ref{assq2} is satisfied for some
$(\beta,\Psi_\beta)$, then it is so for
$(\beta',\Psi_\beta^{\beta/\beta'})$ with any $\beta'>\beta$. In fact,
condition (ii) is ensured due to the concavity of the function $t\mapsto
t^{\beta/\beta'}$ on $[0,\infty)$, while (iii) can be checked using the
H\"older inequality.

%%%%%%%%%%%%%%%%%%%%%%%

%s3 ###
%s3 #&#
\section{Mirror averaging and a PAC-Bayesian bound in
expectation}\label{sec3}

We now introduce the mirror averaging (MA) estimator. First, we fix a
prior $\pi\in\mathcal{P}_{1}(\mcFL)$, a ``temperature'' parameter $\beta>0$, and set
\begin{eqnarray*}
\hat\theta_{m,\llambda}(\bZ)&=& \frac{\exp\{-(1/\beta)
\sum_{i=1}^m Q(Z_i,\flambda)\}} {\int_\Lambda
\exp\{-(1/\beta)\sum_{i=1}^m Q(Z_i,f_w)\} \pi(\mathrm{d} w)},\\
\hat\theta_\llambda&=&\hat\theta_\llambda(\bZ)= \frac1{n+1}\sum_{m=0}^n
\hat\theta_{m,\llambda}(\bZ)
\end{eqnarray*}
with $\hat\theta_{0,\llambda}(\bZ)\equiv1$. For every fixed $\bZ$,
$\hat\theta_\llambda$ is a probability density on $\Lambda$ with
respect to the probability measure $\pi$. Let $\hat\mu_n$ be the
probability measure on $(\Lambda,\mfL)$ having $\hat\theta_\llambda
$ as
density w.r.t. $\pi$. By analogy with the Bayesian context, one can
call $ \hat\theta_\llambda$ and $\hat\mu_n$ the posterior density and
the posterior probability, respectively. Following \cite{jrt08}, where
the case of discrete $\pi$ was considered (see also \cite{l07}), we
define the \textit{MA aggregate}\vadjust{\goodbreak\looseness=-1} as the corresponding posterior mean $\hat
f_n=\sff_{\hat\mu_n}$, that is
$\hat\mu_n(\mathrm{d}\llambda)=\frac1{n+1}\sum_{m=0}^n
\hat\theta_{m,\llambda}(\bZ)\pi(\mathrm{d}\llambda)$ and
%e1 ###
%
%e1 #&#
\begin{equation}\label{f-aggr}
\hat f_n(\bZ,x)=\int_\Lambda\flambda(x)
\hat\theta_\llambda(\bZ)\pi(\mathrm{d}\llambda)
=\frac1{n+1}\sum_{m=0}^n\int_\Lambda\flambda(x)
\hat\theta_{m,\llambda}(\bZ)\pi(\mathrm{d}\llambda).
\end{equation}
To simplify the notation, we suppress the dependence of $\hat f_n$ on
$\bZ$ and $x$ when it causes no ambiguity.

%t1 #&#
\begin{theorem}[(PAC-Bayesian bound in expectation)]\label{Th1}
If Assumptions \textup{\ref{assq1}} and \textup{\ref{assq2}} are fulfilled, then the MA aggregate $\hat
f_n$ satisfies the following oracle inequality
%e2 ###
%
%e2 #&#
\begin{equation}\label{oracle}
\Ex_f[\ell(\hat f_n,f)]\le\inf_{p\in\mcPLambda}\biggl(\int
_\Lambda
\ell(\flambda,f)p(\mathrm{d}\llambda)+\frac{\beta\KL(p,\pi)}{n+1}\biggr),
\end{equation}
where $\KL(p,\pi)$ stands for the Kullback--Leibler divergence
\[
\KL(p,\pi) =
\cases{\displaystyle
\int_\Lambda\log\biggl(\frac{\mathrm{d}p}{\mathrm{d}\pi}(\llambda)\biggr)
p(d\llambda),&\quad  if $p\ll\pi$,\vspace*{2pt}\cr
+\infty, &\quad  otherwise.
}
\]
\end{theorem}

Proof of Theorem \ref{Th1} is given in the \hyperref[appendix]{Appendix}. It is based on a
cancellation argument that can be traced back to
Barron \cite{Barron87}.

The oracle inequality of Theorem \ref{Th1} is in the line of the
PAC-Bayesian bounds initiated in \cite{McAllester} and is applicable
to a large variety of models. Some particularly relevant examples will
be treated in Sections \ref{sec5}--\ref{sec7}. An interesting feature
of Theorem \ref{Th1} is that it is valid for a large class of prior
distributions.

The fact that (\ref{oracle}) holds true for convex mappings $g\mapsto
Q(Z,g)$ has been discussed informally in~\cite{a2}, page 1606, as a
consequence of an oracle inequality for
a randomized estimator. %the informal discussion following Corollary
%5.1 of that paper.
A difference of Theorem \ref{Th1} from the approach in~\cite{a2} is
that the convexity of the loss function is not required.

\begin{remark}
If the cardinality of $\mcFL$ is finite, say $\operatorname{card} (\mcFL)=N$
and $\Lambda=\{1,\ldots,N\}$, inequality (\ref{oracle}) implies
that
\[
\Ex_f[\ell(\hat f_n,f)]\le\min_{j=1,\ldots,
N}\biggl(\ell(f_j,f)+\frac{\beta\log\pi_j}{n+1}\biggr).
\]
Oracle inequalities of this type and similar under different sets of
assumptions were established earlier by several authors (cf. \cite{catbook01,Cat07,y01,y03,Yang04,buno05,jrt08,a2} and the
references therein for closely related results). Our PAC-Bayesian bound
(\ref{oracle}) generalizes the oracle inequality of \cite{jrt08},
Theorem 3.2, to arbitrary, not necessarily finite, family $\mcFL$. In
the settings that we study below, it is crucial to consider uncountable
$\mcFL$. As we will see later, this generalization allows us to take
advantage of sparsity and suggests a powerful alternative to the
classical model selection approach.
\end{remark}

\begin{remark}
For the regression model with additive noise and deterministic design,
PAC-Bayesian bounds in expectation on the empirical $l_2$-norm similar
to (\ref{oracle}) have been obtained in \cite{dt07,dt08,dt09,dt10} for
an EWA, which does not contain the step of averaging. Earlier
\cite{lb06} proved a similar result for the special case of finite
$\operatorname{card}(\mcFL)$ and Gaussian errors. In the notation of the present
paper, the aggregate studied in those works is of the form
$\check{f_n}=\int_{\Lambda} f_{\llambda} \hat\theta_{n,\llambda
}\pi(\mathrm{d}
\llambda) $. Interestingly, in a very recent paper Lecu\'e and
Mendelson \cite{LecueMend10} proved that $\check{f_n}$ does not
satisfy inequality (\ref{oracle}) in the case of i.i.d. observations.
\end{remark}

Finally, we note that the results of this work hold only for proper
priors. However, it is very likely that Theorem~\ref{Th1} extends to
the case of improper priors under some additional assumption ensuring,
for instance, that the integral $\int_\Lambda\exp\{-\frac
1\beta
\sum_{i=1}^m Q(Z_i,f_w)\} \pi(\mathrm{d} w)$ appearing in the definition of
the MA estimator is finite.

%s4 ###
%s4 #&#
\section{Sparsity oracle inequality}\label{sec4}

In this section, we introduce a prior $\pi$ that we recommend to use
for the MA aggregate under the sparsity scenario. Then we prove a
sparsity oracle inequality
(SOI) leading to some natural choices of the tuning parameters of the
prior.

%s4.1 ###
%s4.1 #&#
\subsection{Sparsity prior and SOI}
In what follows, we assume that $\Lambda\subseteq\RR^M$ for some
integer $M\ge2$. We will use bold face letters to denote vectors and,
in particular, the elements of $\Lambda$. We denote by
$\tr(\textsf{A})$ the trace of a square matrix \textsf{A}. To deal with
integrals of the type $\int_\Lambda\ell(\fllambda,f)p(\mathrm{d}\llambda)$,
we introduce the following additional assumption.

{\renewcommand{\theass}{L}\begin{ass}\label{assl}
For every fixed $f\in\mcF$, there exists a
measurable set $\Lambda_0\subset\Lambda$ such that
$\Lambda\setminus\Lambda_0$ has zero Lebesgue measure and the mapping
$L_f \dvtx \Lambda_0\to\RR$, where $L_f(\llambda)= \ell(\fllambda
,f)$, is
twice differentiable. Furthermore, there exists a symmetric $M\times M$
matrix $\mcM$ such that $\mcM-\nabla^2 L_f(\llambda)$ is positive
semi-definite for every $\llambda\in\Lambda_0$, where $\nabla^2
L_f(\llambda)$ stands for the Hessian matrix.
\end{ass}}

We are interested in covering the case of large $M$, possibly much
larger than the sample size $n$. We will be working under the sparsity
assumption, that is, when there exists \mbox{$\llambdastar\in\RR^M$} such that
$f$ is close to $f_{\llambdastar}$ and $\llambdastar$ has a very
small number of nonzero components. We argue that an efficient way for
handling this situation is based on a~suitable
choice of the prior $\pi$. %Here we will use a modification of the
%sparsity prior proposed in \cite{dt08}.
%which coincides with the $M$-fold product of the truncated Student's
%t-distribution with 3 degrees of freedom.
To be more precise, our results will show how to take advantage of
sparsity for the purpose of prediction and not for accurate estimation
of the parameters or selection of the sparsity pattern. Thus, if the
underlying model is sparse, we do not prove that our estimated model is
sparse as well, but we claim that it has a small prediction risk under
very mild
assumptions. %This point constitutes a conceptual difference between
%the procedure of aggregation described below and the procedures
%based on $\ell_1$ techniques such as the LASSO, the Dantzig Selector
%and their analogues.
Nevertheless, we have a numerical evidence that our method can also
recover very accurately the true sparsity pattern \cite{dt09,dt10}. We
observed this in examples where the restrictive assumptions mentioned
in the Introduction are satisfied.
%However, our theoretical results do not
%deal with this property.

Let $\tau$ and $R$ be positive numbers. The {\it sparsity prior} is
defined by
%e3 ###
%
%e3 #&#
\begin{equation}\label{prior}
\pi(d\llambda)=\frac{1}{C_{\tau,R}}\Biggl\{\prod_{j=1}^M
(\tau^2+\lj^2)^{-2}\Biggr\} \1(\|\llambda\|_1\le R)\,\mathrm{d}\llambda,
\end{equation}
where $\|\llambda\|_1=\sum_j |\lj|$ stands for the $\ell_1$-norm,
$\1(\cdot)$ denotes the indicator function, and~$C_{\tau,R}$ is a
normalizing constant such that $\pi$ is a probability density.

The prior (\ref{prior}) has a simple heuristical interpretation. Note
first that $R$ is a regularization parameter, which is typically very
large. So, in a rough approximation we may consider that the factor
$\1(\|\llambda\|_1\le R)$ is almost equal to one. Thus, $\pi$ is
essentially a product of~$M$ rescaled student's distributions.
Precisely, we deal with the distribution of~$\sqrt2\tau\bY$, where~$\bY$ is a random vector with i.i.d. coordinates drawn from
student's $t$ with three degrees of freedom. In the
examples below, we choose a very small $\tau$, smaller than $1/n$.
Therefore, most of the coordinates of $\tau\bY$ are very close to
zero. On the other hand, since student's distribution has heavy tails,
there exists a small portion of coordinates of $\tau\bY$ that are quite
far from zero.

The relevance of heavy tailed priors for dealing with sparsity has been
emphasized by several authors (see \cite{Seeger08}, Section 2.1, and
references therein). Most of this work is focused on logarithmically
concave priors, such as the multivariate Laplace distribution. Also in
wavelet estimation on classes of ``sparse'' functions \cite{JS05} and
\cite{Riv06} invoke quasi-Cauchy and Pareto priors respectively. Bayes
estimators with heavy-tailed priors in sparse Gaussian shift models are
discussed in \cite{AGP07}.

%These heuristics are very clearly seen on the plots presented in
%Figure~\ref{figSimu}, which represents
%the scatter plots of one realization of a random vector in $
%standard Laplace (double exponential) and standard Gaussian
%distributions. The scaling factor is chosen in such a way that the
%probability densities of the
%simulated distributions are equal to $100$ at the origin. Obviously,
%the scatter plot which is
%most likely to behave as a sparse vector is the one corresponding to
%the Student $t(3)$ distribution.
%
%scaled $t(3)$-distribution
%(left panel), Laplace distribution (central panel) and Gaussian
%distribution (right panel). In all three
%cases the location parameter is equal to zero and the scale parameter
%is set to $10^{-2}$.}

We are now in a position to state the SOI for the MA aggregate with the
sparsity prior. The result is even more general because it holds not
only for the MA aggregate but for any estimator satisfying
(\ref{oracle}) with the sparsity prior.

%t2 #&#
\begin{theorem}\label{ThmSparsity}
Let $\hat f_n$ be any estimator satisfying inequality (\ref{oracle}),
where the loss function~$\ell$ satisfies Assumption \ref{assl} and $\pi$ is the
sparsity prior defined as above. Assume that $\Lambda$ contains the set
$B_1(R)=\{\llambda\in\RR^M \mid  \|\llambda\|_1\le R\}$
with $R
> 2M\tau$. Then for all $\llambdastar$ such that $\|\llambdastar\|
_1\le
R-2M\tau$ we have
%e4 ###
%
%e4 #&#
\begin{equation}\label{SOI11}
\Ex_f[\ell(\hat f_n,f)]\le\ell(f_{\llambdastar},f)+
\frac{4\beta}{n+1}\sum_{j=1}^M \log(1+\tau^{-1}|\lj^*|)+{\sf
R}(M,\tau),
\end{equation}
where the residual term is ${\sf R}(M,\tau)=4\tau^2 \tr(\mcM) +
\frac{\beta}{n+1}$.
\end{theorem}

Proof of Theorem \ref{ThmSparsity} is deferred to Section~\ref{sec93}
of the \hyperref[appendix]{Appendix}.

As follows from (\ref{SOI11}), the main term of the excess risk
$\Ex_f[\ell(\hat f_n,f)]-\ell(f_\llambdastar,f)$ is proportional to
$\sum_{j=1}^M \log(1+\tau^{-1}|\lj^*|)$. Importantly, the number of
nonzero elements in this sum is equal to the number of nonzero
components of $\llambdastar$ that we will further denote by
$\|\llambdastar\|_0$. Therefore, for sparse vectors $\llambdastar$ this
term is rather small. But still, in all the examples that we consider
below, it dominates the remainder term ${\sf R}(M,\tau)$, which is made
negligible by choosing a sufficiently small $\tau$, for instance,
$\tau=\mathrm{O}((\tr(\mcM)n)^{-1/2})$.

Theorem~\ref{ThmSparsity} implies the following bound involving only
the $\ell_0$ norm and the upper bound $R$ on the $\ell_1$ norm of
$\llambdastar$.

%c1 #&#
\begin{corollary}
If some estimator $\hat f_n$ satisfies the oracle inequality of
Theorem \ref{ThmSparsity}, then
\[
\Ex_f[\ell(\hat f_n,f)]\le\ell(f_{\llambdastar},f)+\frac{4\beta
\|\llambdastar\|_0\log(1+R\tau^{-1})}{n+1}+{\sf R}(M,\tau),
\]
where $\llambdastar$ and ${\sf R}(M,\tau)$ are as in
Theorem \ref{ThmSparsity}.
\end{corollary}

\begin{pf}
Set $M^*=\|\llambdastar\|_0$ for brevity. Using Jensen's inequality, we
get
\[
\frac1{M^*}\sum_{j=1}^M\log(1+\tau^{-1}|\lj^*|)\le\log
\bigl(1+(\tau
M^*)^{-1}\|\llambdastar\|_1\bigr).
\]
Using the inequalities $\|\llambdastar\|_1\le R$ and $M^*\ge1$, the
desired inequality follows.
\end{pf}

Note that the sparsity oracle inequalities (SOI) stated in this section
are valid not only for the MA aggregate but for any other estimator
(whose definition involves a prior $\pi$) satisfying a PAC-Bayesian
bound similar to (\ref{oracle}), possibly with some additional residual
terms that should be then added in the SOI as well. Examples of such
estimators can be found in \cite{a2}.

%Note also that the SOI provided by Theorem~\ref{ThmSparsity} above is
%in
%the same spirit as \cite[Theorem 5]{dt08}, which is valid for the
%regression
%with deterministic design. However, thanks to a more careful analysis
%the residual terms in Theorem~\ref{ThmSparsity} have more elegant form
%than those in \cite[Theorem 5]{dt08}.

\begin{remark}\label{AssumpL} Assumption \ref{assl}
need not be satisfied exactly. In fact, $L_f(\cdot)$ need not even be
differentiable. Inspection of the proof of Theorem~\ref{ThmSparsity}
reveals that if $L_f(\llambda)$ is well approximated by a smooth
function $\bar L_f(\llambda)$, that is $0\le L_f(\llambda)-\bar
L_f(\llambda)\le\eps$, $\forall\llambda$, for some small $\eps>0$ and
if $\bar\mcM_\eps- \nabla^2 \bar L_f$ is positive semidefinite, then
the conclusions of Theorem~\ref{ThmSparsity} hold with a modified
residual term
\[
{\sf R}_\eps(M,\tau)=\eps+4\tau^2\tr(\bar\mcM_\eps)+\frac
{\beta}{n+1}.
\]
This remark will be useful for studying the problem of classification
under the hinge loss where the function $L_f$ is not differentiable,
cf. Section \ref{sec7}.
\end{remark}

%s4.2 ###
%s4.2 #&#
\subsection{Choice of the tuning parameters}\label{tuning}
The above sparsity oracle inequalities suggest some guidelines for the
choice of tuning parameters $\tau$ and $R$:
\begin{enumerate}
\item Parameter $\tau$ should be chosen very carefully: It should be
small enough to guarantee the negligibility of the residual term but
not exponentially small to prevent the explosion of the main term of
the risk. A reasonable choice (which is not the only possible) for
$\tau$ is
\begin{equation}\label{tau}
\tau=\min\biggl(\frac{\sqrt{\beta}}{\sqrt{Mn}} ,\frac{R}{4M}\biggr) .
\end{equation}
For this choice of $\tau$, we have:
\begin{enumerate}[(a)]
\item[(a)] the residual term ${\sf R}(M,\tau)$ is at most of order
$\beta/n$,
\item[(b)] the terms $\log(1+|\lambda_j^*|/\tau)$ increase at
most logarithmically in $M$ and in $n$ under the condition that
$\tr(\mcM)$ increases not faster than a power of $M$. Note that
$\tr(\mcM)=O(M)$ in all the examples that we consider below.
\item[(c)] the
MA aggregate is accurate enough if there exists a sparse vector
$\llambdastar$, with $\ell_1$-norm bounded by $R/2$ which provides a
good approximation $f_{\llambdastar}$ of $f$.
\end{enumerate}

\item It is clear that one should choose $R$ as large as possible in
order to cover the broadest class of possible values $\llambdastar$.
However, we are not aware of any example where Assumption \ref{assq2} holds with
finite $\beta$ for $R=+\infty$ or, equivalently, for $\Lambda=\RR^M$.
Therefore, we assume that $R$ is an a priori chosen large parameter and
interpret the above results as follows: If there is a sparse vector
$\llambda^*$ such that $\ell(f_{\llambdastar},f)$ is small and
$\|\llambdastar\|_1\le R-2M\tau$, then the MA aggregate has a small
prediction risk.
\end{enumerate}

\begin{remark}
The choice\vspace*{1pt} $\tau=\min(\frac{\sqrt{\beta}}{\sqrt{\tr(\mcM)n}},\frac{R}{4M})$ ensures that the estimator $\hat f_n$ is invariant
with respect to an overall scaling of $\llambda$. More precisely, if
instead of considering the parametrization $\{f_\llambda\dvt
\|\llambda\|_1\le R\}$ we consider the parametrization
$\{\tilde{f_{\bomega}}\dvt \|\bomega\|_1\le R/s\}$ with
$\tilde{f_{\bomega}}=f_{s\bomega}$ for some $s>0$, then the MA
aggregate based on the prior defined by~(\ref{prior}) remains
unchanged. This can be easily checked by the change of variables using
the relation $\tilde\mcM=s^2\mcM$ where $\tilde\mcM$ denotes the
Hessian matrix analogous to $\mcM$ for the dictionary~$\{\tilde{f_{\bomega}}\}$.
\end{remark}

Along with choosing the parameters $(\tau,R)$ of the prior, one needs
to choose the ``temperature'' parameter $\beta$. A model-free choice of
$\beta$ seems to be impossible. In fact, even the existence of $\beta$
such that Assumption \ref{assq2} holds is not ensured for every model. Some more
discussion of the choice of $\beta$ is given in Remark \ref{rem7} below.

%s5 ###
%s5 #&#
\section{Application to regression with random design}\label{sec5}

%s5.1 ###
%s5.1 #&#
\subsection{Regression estimation in $L^2$-norm}
Let $\mcZ=\mcX\times\RR$ and we have the i.i.d. observations
$Z_i=(X_i,Y_i)$, $i=1,\ldots,n$ with $X_i\in\mcX$ and $Y_i\in\RR$. We
define the regression function by $f(x)=\Ex(Y_1|X_1=x)$, $\forall
x\in\mcX$, and assume that the errors
\[
\xi_i=Y_i-f(X_i),\qquad  i=1,\ldots,n,
\]
are such that $\Ex[\xi_1^2]<\infty$. Then $\Ex(\xi_i|X_i)=0$. Let $P_X$
denote the distribution of $X_1$. For $s\in[1,\infty]$, we denote by
$\|\cdot\|_{P_X,s}$ the $L^s$-norm with respect to $P_X$. We also
denote by $\langle\cdot,\cdot\rangle_{P_X}$ to\vadjust{\goodbreak} the scalar product in
$L^2(\mcX,P_X)$. Throughout this section, we consider the integrated
squared loss $\ell(f,g)=\|f-g\|_{P_X,2}^2$. Then it is easy to check
that Assumption~\ref{assq1} is fulfilled with
\[
Q(z,g)=\bigl(y-g(x)\bigr)^2,\qquad  z=(x,y)\in\mcZ.
\]
Furthermore, we focus on the particular case where $\mcFL$ is a convex
subset of the vector space spanned by a finite number of measurable
functions $\{\phi_j\}_{j=1,\ldots,M}\subset L^2(\mcX,P_X)$,
that is
%e5 ###
%
%e5 #&#
\begin{equation}\label{classf}
\mcFL=\Biggl\{\fllambda=\sum_{j=1}^M \lj\phi_j \Big| \llambda
\in
\RR^M\mbox{ with }\|\llambda\|_1\le R\Biggr\}
\end{equation}
for some $R>0$. Then Assumption \ref{assl} holds with $\mcM$ being the matrix
with entries $\langle\phi_j,\phi_{j'}\rangle_{P_X}$, which will be
referred to as the Gram matrix. This definition of $\mcM$ will be used
throughout this section. The collection of functions
$\{\phi_1,\dots,\phi_M\}$ will be called the {\it dictionary}.

\begin{remark}\label{rem4}
The value of the parameter $\tau$ presented in (\ref{tau}) does not
allow us to take into account the possible inhomogeneity of functions
$\phi_j$. One way of dealing with the inhomogeneity is to let $\tau$
depend on $j$ in the definition of the sparsity prior $\pi$. In this
paper, we consider for brevity a less general approach, which is common
in the literature on sparsity. Namely, we normalize the functions
$\phi_j$ in advance and use the same $\tau$ for all coordinates of
$\llambda$. The normalization is done by rescaling the functions
$\phi_j$ so that all the diagonal entries of the Gram matrix $\mcM$ are
equal to one.
\end{remark}

Following this remark, we assume that the functions $\phi_j$ are such
that $\|\phi_j\|_{P_X,2}=1$ for every~$j$. Therefore, $\tr(\mcM)=M$.

%p1 #&#
\begin{proposition}\label{prop1}
Assume that for some constant $L_\phi>0$ we have
$\max_{j=1,\ldots,M}\|\phi_j\|_{P_X,\infty}\le L_\phi$. If, in
addition, the errors $\xi_i$ have a bounded exponential moment:
%e6 ###
%
%e6 #&#
\begin{equation}\label{eq8}
\exists b,\sigma^2>0 \mbox{ such that } \Ex(\mathrm{e}^{t\xi
_1}|X_1)\le
\mathrm{e}^{\sigma^2t^2/2}\qquad \forall|t|\le b,  P_X\mbox{-a.s.},
\end{equation}
then, for every $ \beta\ge
\max(2\sigma^2+2\sup_{\llambda\in\Lambda}\|\fllambda-f\|
_{P_X,\infty}^2,4RL_\phi/b)$,
the MA aggregate $\hat f_n$ defined by (\ref{f-aggr}) with the sparsity
prior (\ref{prior}) satisfies
%e7 ###
%
%e7 #&#
\begin{equation}\label{SOI1}
\Ex_f[\|\hat f_n-f\|_{P_X,2}^2]\le\inf_{\llambdastar}
\Biggl\{\|f_{\llambdastar}-f\|_{P_X,2}^2+\frac{4\beta}{n+1}\sum_{j=1}^M
\log(1+\tau^{-1}|\lj^*|)\Biggr\}+{\sf R}(M,\tau),
\end{equation}
where the $\inf$ is taken over all $\llambdastar$ such that
$\|\llambdastar\|_1\le R-2M\tau$ and ${\sf R}(M,\tau)=4\tau^2
M+\frac{\beta}{n+1}$.
\end{proposition}
\begin{pf*}{Proof of Proposition~\ref{prop1}}
In view of Theorem~\ref{ThmSparsity}, it suffices to check that
Assumption~\ref{assq2} is fulfilled for $\beta\ge
\max(2\sigma^2+2\sup_{\llambda\in\Lambda}\|\fllambda-f\|
_{P_X,\infty}^2,4RL_\phi/b)$.
This is done along the lines of the proof of \cite{jrt08}, Corollary~5.5. We omit the details.
\end{pf*}

Proposition \ref{prop1} can be used in signal denoising under the
sparsity assumption. A typical issue studied in statistical literature,
as well as in the literature on signal processing, is to estimate a
signal $f$ based on its noisy version recorded at some points
$X_1,\ldots,X_n$, under the assumption that $f$ admits a sparse
representation w.r.t. some given dictionary $\{\phi_j ;
j=1,\ldots,M\}$. By sparse representation, we mean a linear
combination of a~small number of functions $\phi_j$. Assume for the
moment that the noise satisfies (\ref{eq8}) with $b=+\infty$ and some
known $\sigma\in[0,\infty)$ and that the unknown signal is bounded by
some constant that can be assumed to be equal to 1. The latter
assumption is fulfilled in many applications, as for example in image
processing.

The method that we suggest for estimating a sparse representation of
$f$, under the assumption $M\ge n$, consists of:
\begin{enumerate}
\item[(a)] normalizing the functions $\phi_j$,
\item[(b)] fixing a
parameter $R>0$,
\item[(c)] setting
%e8 ###
%
%e8 #&#
\begin{equation}\label{paramm}
\beta=2\sigma^2+2(RL_\phi+1)^2,\qquad
\tau=\min\biggl(\frac{\sqrt{\beta}}{\sqrt{\tr(\mcM)n}},\frac{R}{4M}\biggr),
\end{equation}
\item[(d)] computing the MA aggregate $\hat f_n=\sum_{j=1}^M
\hat\lambda_{ j}\phi_j$ with coefficients
$\hat\lambda_{ j}=\int_{\RR^M}\lj\hat\theta_\llambda\pi
(\mathrm{d}\llambda)$
based on the sparsity prior (\ref{prior}) and the posterior density
\[
\hat\theta_\llambda=\frac1{n+1}\sum_{m=0}^{n+1}
\frac{\exp\{-(1/\beta)\sum_{i=1}^m (Y_i-\fllambda
(X_i))^2\}}
{\int_\Lambda\exp\{-(1/\beta)\sum_{i=1}^m
(Y_i-f_{\mathbf{w}}(X_i))^2\}\pi(\mathrm{d}\mathbf{w})} .
\]
\end{enumerate}

In view of Proposition~\ref{prop1}, if we run this procedure with some
value $R>0$, we will get accurate estimates for signals that are well
approximated by a sparse linear combination of functions $\phi_j$,
provided that the coefficients of this linear combination have an
$\ell_1$-norm bounded by $R-2M\tau$. In most of the problems arising in
signal or image processing the $\ell_1$-norm of the best sparse
approximation to the signal is unknown. It is therefore important to
make a data-driven choice of $R$. Let us outline one possible way to do
this. Consider that only the signals formed by a linear combination of
at most $M^*$ functions $\phi_j$ are of interest, and assume that the
dictionary $\{\phi_j\}$ satisfies the restricted isometry
property (RIP) of order $M^*$, see equation (1.3) in \cite{ct} for the
definition. In other terms, assume that $f\approx f_{\llambdastar}$
with $\|\llambdastar\|_0\le M^*$ and $\|f_{\llambdastar}\|_{P_X,2}\ge
\frac12\|\llambdastar\|_2$ where $\|\cdot\|_2$ is the Euclidean norm.
Then we can bound the $\ell_1$-norm of $\llambdastar$ as follows:
\[
\|\llambdastar\|_1\le\sqrt{M^*}\|\llambdastar\|_2\le
2\sqrt{M^*}\|f_{\llambdastar}\|_{P_X,2}\approx
2\sqrt{M^*}\|f\|_{P_X,2}.
\]
We can estimate $\|f\|_{P_X,2}^2$ consistently by $\frac1n\sum_{i=1}^n
(Y_i^2-\sigma^2)$. Based on these estimates, we suggest the following
data-driven choice of $R$:
\[
\hat R=4\Biggl[\frac{\hat M^*}{n}\sum_{i=1}^n
(Y_i^2-\sigma^2)\Biggr]_+^{1/2},
\]
where $x_+=\max(x,0)$ and $\hat M^*$ a prior approximation of the
sparsity index of the signal~$f$.

\begin{remark}
The choice of $\beta$ in (\ref{paramm}) requires the knowledge of
$\sigma^2$, which characterizes the magnitude of the noise. This value
may not be available in practice. Then it is natural to consider
$\beta$ as a tuning parameter and to select it by a data-driven method,
for example, by a suitably adapted version of cross-validation. This
point deserves a special attention and is beyond the scope of the
present paper.\vspace*{-2pt}
\end{remark}

\begin{remark}\label{rem7}
If the distribution $P_X$ of the design is unknown, it is impossible to
normalize the dictionary functions $\phi_j$. In such a situation,
that is, when the functions $\phi_j$ do not necessarily
satisfy $\|\phi_j\|_{P_X,2}=1$, the claim of Proposition~\ref{prop1}
continues to hold true with the modified residual term ${\sf
R}(M,\tau)=4\tau^2\tr(\mcM)+ \frac{\beta}{n+1}$, which can be bounded
by $4\tau^2 M L_\phi^2+ \frac{\beta}{n+1}$. Thus, once again, choosing
$\tau$ as in (\ref{paramm}) makes the residual term ${\sf R}(M,\tau)$
negligible w.r.t. the main terms of the risk bound.\vspace*{-2pt}
\end{remark}

\begin{remark} Proposition~\ref{prop1} is in agreement with the main
principles of the theory of compressive sampling and sparse recovery,
cf., for example,  \cite{Candes06}. Indeed, if the tuning parameters
are well-chosen, the prediction done by $\hat f_n$ can be quite
accurate even if the sample size is relatively small with respect to
the dimension $M$. This happens if the signal admits a $M^*$-sparse
representation in a possibly overcomplete dictionary of cardinality
$M$. Then the number of observations sufficient for an accurate
prediction is of order $M^*$ up to a logarithmic factor.
Proposition~\ref{prop1} is also in perfect agreement with the principle
of incoherent sampling (see, for instance, \cite{Candes06}, page 10).
In fact, in our setting, the incoherence of the sampling is ensured by
the fact that $\phi_j\in L^2(\mcX,P_X)$ satisfy $\|\phi_j\|_{P_X,2}=1$.\vspace*{-2pt}
\end{remark}

Before closing this section, let us mention the recent
work \cite{Gaiffas}, where some interesting results on the aggregation
of estimators in sparse regression are obtained.\vspace*{-2pt}

%s5.2 ###
%s5.2 #&#
\subsection{Linear regression with random design}\vspace*{-2pt}\label{Reglin}
Consider now the case
of linear regression. Assume that the i.i.d. observations
$(\bX_i,Y_i)$, $i=1,\ldots,n$, are drawn from the linear model
%e9 ###
%
%e9 #&#
\begin{equation}\label{linmodd}
Y_i=\bX_i^\T\llambdastar+\xi_i,\qquad  i=1,\ldots,n,
\end{equation}
where $\bX_i\in\RR^M$ are i.i.d. covariates and
$\llambdastar\in\RR^M$ is the parameter of interest. Then our method
reduces to estimating $\llambdastar$ by
\[
\hat{\llambda}_n=\frac1{n+1}\sum_{m=0}^{n+1}\int_{\RR^M}\llambda
\hat\theta_{m,\llambda}\pi(\mathrm{d}\llambda),
\]
where $\pi$ is the sparsity prior and
\[
\hat\theta_{m,\llambda}=\frac{\exp\{-\beta^{-1}\sum_{i=1}^m
(Y_i-\bX_i^\T\llambda)^2\}}
{\int_{\RR^M}\exp\{-\beta^{-1}\sum_{i=1}^m (Y_i-\bX_i^\T
\bomega)^2\}\pi(\mathrm{d}\bomega)}.
\]
Then the following result holds.\vadjust{\goodbreak}

%p2 #&#
\begin{proposition}\label{regprop}
Consider the linear model (\ref{linmodd}) satisfying the above
assumptions. Let the support of the probability distribution of $\bX_1$
be included in $[-1,1]^M$ and $\Ex[\mathrm{e}^{t\xi_1}|\bX_1]\le \mathrm{e}^{\sigma^2
t^2/2}$ for all $t\in\RR$. Set $\Sigma_X=\Ex[\bX_1\bX_1^\T]$.
Then for
any $\beta\ge2\sigma^2+2(R+\|\llambdastar\|_1)^2$ and any~$\llambdastar$ such that $\|\llambdastar\|_1\le R- 2M\tau$ we have
%e10 ###
%
%e10 #&#
\begin{equation}\label{regc}
\Ex[\|\Sigma_X^{1/2}(\hat\llambda_n-\llambdastar)\|_2^2]\le
\frac{\beta}{n+1}\Biggl(1+4\sum_{j=1}^M
\log(1+\tau^{-1}|\lj^*|)\Biggr)+4\tau^2 \tr(\Sigma_X).
\end{equation}
\end{proposition}

This proposition follows directly from Proposition~\ref{prop1} by
setting $\phi_j({\boldsymbol x})=x_j$ if $|x_j|\le1$ and $\phi
_j({\boldsymbol x})=0$ if
$|x_j|>1$, where ${\boldsymbol x}\in\RR^M$ and $x_j$ is its $j$th
coordinate. Note
also that here we have $\mcM=\Sigma_X$.\vspace*{-3pt}

%An interesting consequence of Proposition~\ref{regprop} is related to
%noisy compressed sensing where the i.i.d. design $\bX_1,\dots, \bX_n$
%is to be chosen in order to guarantee accurate reconstruction of sparse
%vectors $\llambdastar$. When the $\ell_1$ based techniques are used,
%reasonable results can be obtained only for subgaussian distribution of
%$\bX_1$, cf. e.g., \cite{Candes06}. Proposition~\ref{regprop} shows
%that for the MA aggregate we have a more general result. In particular,
%it implies that for any $\bX_1$ with i.i.d. coordinates $X_{i1}$ such
%that $\Ex(X_{i1})=0$ and $\Ex(X_{i1}^2)=1$ the squared risk
%$ \Ex[\|\hat\llambda_n-\llambdastar\|^2_2]$ is bounded from above by
%the right hand side of~(\ref{regc}).

%s5.3 ###
%s5.3 #&#
\subsection{Rate optimality}\label{optimality}
\vspace*{-3pt}

In this section, we discuss the optimality of the rates of aggregation
obtained in Proposition~\ref{prop1}. We show that the MA aggregate with
the sparsity prior attains, up to a logarithmic factor, the optimal
rates of aggregation (cf. \cite{tsy03}). Furthermore, $\hat f_n$ is
adaptive in the sense that it simultaneously achieves the optimal rates
for the Model Selection~(MS), Convex~(C) and Linear (L) aggregation. In
what follows, these rates are denoted respectively by $\psi_n^{\mathbf{MS}}(M)$, $\psi_n^{\mathbf{C}}(M)$ and $\psi_n^{\mathbf{L}}(M)$. It is established
in \cite{tsy03} that:
\begin{eqnarray*}
\psi_n^{\mathbf{MS}}(M)&=&n^{-1}{\log M},\\[-2pt]
\psi_n^{\mathbf{C}}(M) &=&n^{-1}\bigl(M\wedge\sqrt{n}\bigr)\log(1+Mn^{-1/2}),\\[-2pt]
\psi_n^{\mathbf{L}}(M) &=&n^{-1}M.
\end{eqnarray*}
We wish to compare the risk of the estimator $\hat f_n$ with the
sparsity prior $\pi$ to the smallest error
$\|f_\llambdastar-f\|_{P_X,2}^2$ where $\llambdastar$ is one of
$\llambda^{\mathbf{MS}}$, $\llambda^{\mathbf{C}}$ or $\llambda^{\mathbf{L}}$ such that
\begin{eqnarray*}
\llambda^{\mathbf{MS}}&=&\operatorname{arg}\min_{\|\llambda\|_0=\|\llambda\|
_1=1} \|f_\llambda-f\|_{P_X,2}^2,\\[-2pt]
\llambda^{\mathbf{C}}&=&\operatorname{arg}\min_{\|\llambda\|_1\le1} \|f_\llambda
-f\|_{P_X,2}^2,\\[-2pt]
\llambda^{\mathbf{L}}&=&\operatorname{arg}\min_{\llambda\in\RR^M}
\|f_\llambda-f\|_{P_X,2}^2.
\end{eqnarray*}
In the next proposition, we denote by $c$ constants which do not depend
on $M$ and $n$.\vspace*{-2pt}
%
%p3 #&#
\begin{proposition}
Assume that $\hat f_n$ satisfies (\ref{SOI1}) with some $\beta>0$
independent of $M$ and $n$, and that $\log(M)\le c_0 n$ for some
constant $c_0$. If $R>4$ and $\tau$ satisfies (\ref{tau}) with
$\tr(\mcM)=M$, then
\[
\Ex_f[\|\hat f_n-f\|_{P_X,2}^2]\le\|f_{\llambda^{\mathbf{MS}}}-f\|_{P_X,2}^2+c\psi_n^{\mathbf{MS}}(M)\log(1+nM)
\]
and
\[
\Ex_f[\|\hat f_n-f\|_{P_X,2}^2]\le\|f_{\llambda^{\mathbf{C}}}-f\|_{P_X,2}^2+c\psi_n^{\mathbf{C}}(M)\log(1+nM).\vadjust{\goodbreak}
\]
Finally, if $\|\llambda^{\mathbf{L}}\|_1\le R-2M\tau$, then
\[
\Ex_f[\|\hat f_n-f\|_{P_X,2}^2]\le\|f_{\llambda^{\mathbf{L}}}-f\|_{P_X,2}^2+c\psi_n^{\mathbf{L}}(M)\log(1+nM).
\]
\end{proposition}

\begin{pf}
For model selection and linear aggregation, the result follows
immediately from~(\ref{SOI1}) by putting there
$\llambdastar=\llambda^{\mathbf{MS}}$ or $\llambdastar=\llambda^{\mathbf{L}}$ and
using that $\|\llambda^{\mathbf{MS}}\|_0=\|\llambda^{\mathbf{MS}}\|_1=1$. The
case of convex aggregation with $M\le\sqrt{n}$ follows from the bound
for the linear aggregation. The case $M>\sqrt{n}$ requires some
additional arguments, which are presented below.\looseness=-1

Let $s=s_n$ be the integer part of $\sqrt n$, denoted by $[\sqrt n]$.
We assume that $\llambda^{\mathbf{C}}$ has at least $s_n$ nonzero
coordinates, the case $\|\llambda^{\mathbf{C}}\|_0 <[\sqrt{n}]$ being a
trivial consequence of (\ref{SOI1}). Using the Maurey randomization
argument as in \cite{BRT08,RT10}, one can show that
%e11 ###
%
%e11 #&#
\begin{equation}\label{Maurey}
\mathop{\min_{\|\llambda\|_1\le1}}_{\|\llambda\|_0\le s}
\|\fllambda-f\|_{P_X,2}^2\le\|f_{\llambda^{\mathbf{C}}}-f\|_{P_X,2}^2+
\frac{\|\llambda^{\mathbf{C}}\|_1^2}{\min(s,\|\llambda^{\mathbf{C}}\|_0)}
\le\|f_{\llambda^{\mathbf{C}}}-f\|_{P_X,2}^2+\frac{1}{s}.
\end{equation}
Let $\llambda^{s,\mathbf{C}}$ be a point where the minimum on the left-hand
side of (\ref{Maurey}) is attained. Since~$\llambda^{s,\mathbf{C}}$ has not
more than $s$ nonzero coordinates and $\|\llambda^{s,\mathbf{C}}\|_1 \le1$,
we have $\sum_j \log(1+|\lambda^{s,\mathbf{C}}_j/\tau|)\le
s\log(1+\|\llambda^{s,\mathbf{C}}\|_1/\tau)\le s\log(1+\tau^{-1})$. Thus,
applying (\ref{SOI1}) to $\llambda^*=\llambda^{s,\mathbf{C}}$ and using
(\ref{tau}), we get
\begin{equation}\label{13}
\Ex_f[\|\hat f_n-f\|^2_{P_X,2}]\le\|f_{\llambda^{s,\mathbf{C}}}-f\|_{P_X,2}^2
+\frac{cs\log(1+\tau^{-1})}{n},
\end{equation}
where $c$ is some constant independent of $n$ and $M$. Recall now that
$\|f_{\llambda^{s,\mathbf{C}}}-f\|_{P_X,2}^2$ is equal to the left-hand side
of (\ref{Maurey}). This implies
\[
\Ex_f[\|\hat f_n-f\|^2_{P_X,2}]\le\|f_{\llambda^{\mathbf{C}}}-f\|_{P_X,2}^2+\frac{1}{s}+\frac{cs\log(1+\tau^{-1})}{n},
\]
which leads to the desired result due to the choice $s=[\sqrt{n}]$ and
(\ref{tau}).
\end{pf}

\begin{remark}
The theory developed here relies on the fact that the risk is measured
by the expected squared loss. In the case of general $L_p$-loss with
$p\ge1$, a universal procedure for aggregation is proposed in
\cite{Golden} and it is proved that the aggregation in $L_p$ for $p>2$
is more difficult than it is in $L_2$.
\end{remark}

%s6 ###
%s6 #&#
\section{Application to density estimation}\label{sec6}

Let $X_1,\ldots,X_n$ be the observations, which are independent copies
of a random variable $X\dvtx\Omega\to\mcX$ whose distribution has a density
$f$ with respect to some reference measure $\mu$. We consider the
problem of estimating $f$ based on $X_1,\ldots,X_n$. We measure the
risk of an estimator $\tilde f$ of $f$ by the integrated squared error
\[
\ell(\tilde f,f)=\|\tilde f-f\|_{\mu,2}^2=\int_\mcX\bigl(\tilde
f(x)-f(x)\bigr)^2\mu(\mathrm{d}x).
\]
Define the mapping $Q(\cdot,g)\dvtx\mcX\times L^2(\mcX,\mu)\to\RR$ by
\[
Q(x,g)=\|g\|_{\mu,2}^2-2g(x).
\]
It is straightforward that $\Ex_f Q(X,g)-\ell(g,f)=-\|f\|_{\mu,2}^2$
and, therefore, Assumption \ref{assq1} is fulfilled. To further specify the
setting, we consider the family $\mcF_\Lambda$ defined in
(\ref{classf}) where the functions $\phi_j$ are chosen from
$L^2(\mcX,\mu)$ so that $\|\phi_j\|_{\mu,2}=1$ and
$\|\phi_j\|_{\mu,\infty}\le L$, $j=1,\ldots,M$, for some positive
constant $L$. Note that the functions $\phi_j$ need not be integrable
or positive. We have the following result.

%p4 #&#
\begin{proposition}\label{propdensity}
Let the assumptions given above in this subsection be satisfied and
$\|f\|_{\mu,\infty}\le L$. If $\beta$ is such that
\begin{equation}\label{betaa}
(\beta-2R^2)\mathrm{e}^{-4R(L+\sqrt{L})/\beta}\ge2L+4R L,
\end{equation}
then the MA aggregate $\hat f_n$ based on the sparsity prior
(\ref{prior}) satisfies
%e12 ###
%
%e12 #&#
\begin{equation}\label{SOI1den}
\Ex_f[\|\hat f_n-f\|_{\mu,2}^2]\le\inf_{\llambdastar}
\Biggl\{\|f_{\llambdastar}-f\|_{\mu,2}^2+\frac{4\beta}{n+1}\sum_{j=1}^M
\log(1+\tau^{-1}|\lj^*|)\Biggr\}+{\sf R}(M,\tau),
\end{equation}
where the $\inf$ is taken over all the vectors $\llambdastar$ such that
$\|\llambdastar\|_1\le R-2M\tau$.
\end{proposition}

The proof of this proposition is given in the \hyperref[appendix]{Appendix}. It consists in
checking that Assumptions \ref{assq2} and \ref{assl} are satisfied and then applying
Theorem \ref{ThmSparsity}.
%essentially the same as that of \cite[Corollary 5.7]{jrt08}.
Condition (\ref{betaa}) can be significantly simplified in many
concrete situations. For example, if we assume that $R=1$ or $R=2$,
then one can choose $\beta=12 L$ and $\beta=23L$ respectively, provided
that $L\ge2$.

%s7 ###
%s7 #&#
\section{Classification}\label{sec7}

Assume that we have a sample $(X_1,Y_1),\ldots,(X_n,Y_n)$, where
$X_i\in\mcX$ and $Y_i\in\{-1,+1\}$ are labels. Here $\mcX$ is an
arbitrary measurable space and $(X_i,Y_i)$
are assumed to be generated independently according to a
probability distribution $P$. The goal of binary classification
is to assign a
label $+1$ or $-1$ to a new random point $x$ which is distributed as
$X_i$ and independent of $X_1,\dots,X_n$.

%t1 ###
%
%t1 #&#
\begin{table}[b]
\caption{Common choices of function
$\Phi$; classifiers $f_\Phi$ minimizing the $\Phi$-risk; the
corresponding functions $Q$; constants $\beta_\Phi$ and $C_\Phi$
appearing in Proposition \protect\ref{classif1}}\label{TabPhi}
\begin{tabular*}{\textwidth}{@{\extracolsep{\fill}}llllll@{}}
\hline
Loss & $\Phi(u)$ & $f_\Phi(x)$ & $Q(z,g)$ & $\beta_\Phi$ & $C_\Phi$ \\
\hline
 Squared & $(1+u)^2$ & $\eta(x)$ &
$(1-yg(x))^2$ & $2(1+RL_\Phi)^2$ & 8\\
Truncated squared & $\{(1+u)_+\}^2$ &
$\eta(x)$ & $\{\max(1-yg(x),0)\}^2$ & $2(1+RL_\Phi)^2$ & 8\\ [2pt]
Boosting & $e^u$ & $\frac12\log\frac
{1+\eta(x)}{1-\eta(x)}$ & $e^{-yg(x)}$ & $e^{RL_\Phi}$ &
$4e^{RL_\Phi}$\\ [3pt]
Logit-boosting & $\log(1+e^u)$ & $\log
\frac{1+\eta(x)}{1-\eta(x)}$ & $\log(1+e^{-yg(x)})$ & $e^{RL_\Phi}$ & $4$ \\
Misclassification & $\1(u=1)$ & $\eta(x)$ & $\1
(g(x)\not=y)$ & -- & -- \\
Hinge & $(1-u)_+$ & $\eta(x)$ & $\max
(1-yg(x),0)$ & -- & -- \\
\hline
\end{tabular*}
\end{table}

The problem of interest in classification is to design a classifier
$\hat f\dvtx\mcX\to\RR$ having a small misclassification risk $R[\hat
f]=\int_{\mcX\times\{-1,+1\}}\1(\sign(\hat f(x))\not=y)P(\mathrm{d}x,\mathrm{d}y)$.
Denote by $\eta\dvtx\mcX\to[-1,1]$ the regression function
\[
\eta(x)=\Ex(Y_1|X_1=x)=2\Pb(Y_1=1|X_1=x)-1\qquad \forall x\in
\mcX.
\]
The Bayes classifier is defined as follows:
$f(x)=\1(\eta(x)>0)-\1(\eta(x)\le0)=\sign(\eta(x))$. One easily checks
that
\[
R[\hat f]-R[f]=\int_\mcX\1\bigl(\sign(\hat f(x))\not=
f(x)\bigr)|\eta(x)|P_X(\mathrm{d}x),
\]
where $P_X$ is the distribution of $X_1$. This shows that the Bayes
classifier $f$ minimizes the misclassification risk. Clearly, the Bayes
classifier is not available in practice because of its dependence on
the unknown regression function $\eta(\cdot)$.

This problem is a special case of the general setting of
Section~\ref{sec2} if we take there $Z_i=(X_i,Y_i)$ and
$\ell(g,f)=R[g]-R[f]$. Assumption \ref{assq1} is then fulfilled with
$Q(z,g)=\1(\sign(g(x))=y)$ where $z=(x,y)$. However,
Assumptions \ref{assq2} and \ref{assl} are not satisfied.

%s7.1 ###
%s7.1 #&#
\subsection{\texorpdfstring{Classification under smooth $\Phi$-losses}{Classification under smooth Phi-losses}}
An alternative approach is to consider the $\Phi$-risk of classifiers.
For a fixed convex twice differentiable function $\Phi\dvtx\RR\to\RR_+$,
the $\Phi$-risk of a classifier $\hat f$ is defined by
\begin{eqnarray*}
R_\Phi[\hat f]&=&\int_{\mcX\times\{\pm1\}} \Phi(-y\hat
f(x))P(\mathrm{d}x,\mathrm{d}y)\\
&=& \frac12\int_\mcX\bigl\{\Phi(-\hat
f(x))\bigl(1+\eta(x)\bigr)+\Phi(\hat
f(x))\bigl(1-\eta(x)\bigr)\bigr\}P_X(\mathrm{d}x).
\end{eqnarray*}
In this subsection, we are mainly interested in the four common choices
of $\Phi$ presented in the top lines of Table~\ref{TabPhi}. For these
and other loss functions, sharp relations between the $\Phi$-risk and
the misclassification risk of a given classifier $\hat f$ have been
established in \cite{Zhang04,Bartlett06}. In particular, it
is proved in these papers that the minimum of $\Phi$-risk is attained
at any classifier satisfying
\[
f_\Phi(x)\in\operatorname{arg}\min_{u\in\RR}
\bigl\{\Phi(-u)\bigl(1+\eta(x)\bigr)+\Phi(u)\bigl(1-\eta(x)
\bigr)\bigr\}\qquad
\forall x\in\mcX.
\]
Note however that in practice the computation of $f_\Phi$ is impossible
because of its dependence on the unknown $\eta$.

Our aim here is to design a classifier having a $\Phi$-risk which is
nearly as small as the minimal possible $\Phi$-risk. This task can be
recast in a problem of estimation where $f_\Phi$ is the function to be
estimated and the quality of an estimator (classifier) $\hat f$ is
measured by the excess\vadjust{\goodbreak} risk $R_\Phi[\hat f]-R_\Phi[f_\Phi]$. Therefore,
this is a particular case of the setting described in
Section \ref{sec2} with $\ell(g,f)=\ell_\Phi(g,f_\Phi)=
R_\Phi[g]-R_\Phi[f_\Phi]$ and $Q(z,g)=\Phi(-yg(x))$ for every
$z=(x,y)$. Here Assumption \ref{assq1} is obviously satisfied.

In the same spirit as in the previous sections, we assume that we are
given a dictionary $\{\phi_j\}_{j=1,\ldots,M}$ of functions on $\mcX$
with values in $\RR$. The family $\mcFL$ is defined as the set of all
linear combinations of the functions $\phi_j$ with coefficients
$\lambda_1,\ldots,\lambda_M$, such that the vector
$\llambda=(\lambda_1,\ldots,\lambda_M)$ belongs to the $\ell_1$ ball
with radius $R$, cf. (\ref{classf}). The next proposition shows that a
strong sparsity oracle inequality holds for an appropriate choice of
$\beta$.\vspace*{-3pt}

%p5 #&#
\begin{proposition}\label{classif1}
Assume that for some constant $L_\phi>0$ we have
$\max_{j=1,\ldots,M}\|\phi_j\|_{P_X,\infty}\le L_\phi$. Let the
function $\Phi$ be twice continuously differentiable with\footnote{We
use here the convention $0/0=0$.}
\[
\beta_\Phi:=\sup_{|u|\le RL_\phi} \frac{\Phi'(u)^2}{\Phi
''(u)}<\infty.
\]
Then the MA aggregate defined with $\beta\ge\beta_\Phi$ and with the
sparsity prior (\ref{prior}) satisfies
%e13 ###
%
%e13 #&#
\begin{eqnarray}\label{SOI-classif}
\Ex_f[\ell_\Phi(\hat f_n,f)]&\le&\min_{\|\llambdastar\|_1\le
R-2M\tau}\Biggl(
\ell_\Phi(f_{\llambdastar},f)+\frac{4\beta}{n+1}\sum_{j=1}^M
\log(1+\tau^{-1}|\lambdastar_j|)\Biggr)\nonumber\\ [-8.5pt]\\ [-8.5pt]
&&{}+C_\Phi\tau^2 \sum_{j=1}^M
\|\phi_j\|_{P_X,2}^2+ \frac{\beta}{n+1},\nonumber
\end{eqnarray}
where $C_\Phi=4\max_{|u|\le R L_\phi}\Phi''(u)$.\vspace*{-3pt}
\end{proposition}

\begin{pf}
We apply Theorems~\ref{Th1} and~\ref{ThmSparsity}. First, we show that
Assumption \ref{assq2} is satisfied. Recall that $Q(z,g)=\Phi(-yg(x))$ and set
$\Psi_\beta(g,\tilde g)=\int_{\mcX\times\{\pm1\}}
\exp(-\beta^{-1}\{Q(z,g)-Q(z,\tilde g)\})P(\mathrm{d}x,\mathrm{d}y)$.
Let us show that for $\beta\ge\beta_\Phi$ the mapping
$g\mapsto\Psi_\beta(g,\tilde g)$ is concave. By standard arguments,
this reduces to proving that the function $t\mapsto
\phi(t)=\Psi_\beta(tg+(1-t)\bar g,\tilde g)$ is concave on $t\in[0,1]$
for every fixed $g,\bar g$ and $\tilde g$. A simple algebra shows that
the second derivative of $\phi$ is nonpositive on $[0,1]$ whenever
$\beta\ge\Phi'(-yg(x))^2/\Phi''(-yg(x))$ for all
$(x,y)\in\mcX\times\{\pm1\}$ and all $g\in\mcFL$. On this set of
$x,y,g$ the value $-yg(x)$ belongs to the interval
$[-RL_\phi,RL_\phi]$. Thus, Assumption \ref{assq2} is satisfied for $\beta\ge
\beta_\Phi$ and Theorem~\ref{Th1} can be applied.

To use Theorem~\ref{ThmSparsity}, it remains to prove that Assumption
\ref{assl}
is satisfied with $\mcM$ being the matrix with entries
$(\frac14C_\Phi\langle\phi_j,\phi_{j'}\rangle)$, where
$j$ and
$j'$ run over $\{1,\ldots,M\}$. From the formula for $R_\Phi[\hat f]$
given at the beginning of this subsection we get
\[
\nabla^2 L_f(\llambda)=\nabla^2 R_\Phi[\fllambda]=\int_{\mcX
\times\{\pm 1\}}
\bigl(\nabla\fllambda(x)\cdot\nabla\fllambda(x)^\T\bigr)\Phi
''(-y\fllambda(x))
P(\mathrm{d}x,\mathrm{d}y).
\]
Since $y\fllambda(x)\in[-RL_\phi,RL_\phi]$ the matrix $\mcM-\nabla^2
L_f(\llambda)$, where $\mcM=\frac14 C_\Phi\int(\nabla\fllambda
\nabla
\fllambda^\T)(x)\*P_X(\mathrm{d}x)$, is positive semi-definite. The desired
result follows now from the linearity in $\llambda$ of $\fllambda(x)$.\vadjust{\goodbreak}
\end{pf}

For the four common choices of $\Phi$ presented in the top lines of
Table \ref{TabPhi} all the conditions of Proposition~\ref{classif1} are
satisfied for a properly chosen constant $\beta$. The minimal values of
$\beta$, as well as the values of the constant $C_\Phi$, for each loss
function $\Phi$ are reported in the last two columns of
Table~\ref{TabPhi}. It is often interesting to use binary classifiers
$\phi_j$ (i.e., functions with values in $\{\pm1\}$), in
which case $L_\phi=1$. Also note that the expressions for
$\beta_{\Phi}$ suggest to choose $R$ not too large, especially in the
case of the boosting and the logit-boosting losses.\looseness=-1

%s7.2 ###
%s7.2 #&#
\subsection{Classification under the hinge loss}
One of the key issues in machine learning is classification by support
vector machines. They correspond to a penalized $\Phi$-risk
classification with the loss $\Phi(u)=\Phi_H(u)=\max(1+u,0)$, referred
to as the hinge loss. A notable feature of the hinge loss is that the
classifier $f_{\Phi_H}(x)$ equals $\operatorname{sgn}(\eta(x))$ and therefore
coincides with the Bayes classifier for the misclassification risk.
However, since the hinge loss does not satisfy Assumptions \ref{assq2} and \ref{assl},
Proposition~\ref{classif1} cannot be applied. Furthermore, as shown in
\cite{Lecue07}, no aggregation procedure can attain the fast rate of
aggregation (i.e., the rate $1/n$ up to a logarithmic factor)
when the risk is measured by the hinge loss.

The reason for the failure of Assumption \ref{assl} is that the hinge loss is
not continuously differentiable. One can circumvent this problem by
using the smoothing argument of Remark~\ref{AssumpL}. Indeed, let us
fix $\eps>0$ and introduce the function
$K_\eps(z)=(\sqrt{\eps^2+z^2}-\eps)\1(z>0)$, which is a smooth
approximation to the positive part of $z$. It is easy to see that
$K_\eps(z)\le\max(z,0)\le K_\eps(z)+\eps$ and that
$K_\eps''(z)=\eps^2(\eps^2+z^2)^{-3/2}\in(0,\eps^{-1}]$ for $z>0$. This
allows us to approximate the loss $\ell_{\Phi_H}(g,f)$ by
\[
\ell_\eps(g,f)=\frac12\int_\mcX\bigl\{K_\eps\bigl(1-g(x)\bigr)\bigl(1+\eta(x)\bigr)
+K_\eps\bigl(1+g(x)\bigr)\bigl(1-\eta(x)\bigr)\bigr\}P_X(\mathrm{d}x)-R_{\Phi_H}[f].
\]
Although Assumption \ref{assq2} is not fulfilled, the next proposition shows
that it is possible to adapt the argument of Proposition \ref{classif1}
to the hinge loss $\Phi=\Phi_H$. However, unlike
Proposition \ref{classif1} where the rate of convergence is of the
order $1/n$ (up to a logarithmic factor), the resulting sparsity oracle
inequality has only the rate $1/\sqrt{n}$ (up to a logarithmic factor),
cf. also Remark \ref{Remark9} (1) below. This is the best we can get
for the hinge loss without imposing any condition on $\eta$.

%p6 #&#
\begin{proposition}\label{classif2}
Let $\Phi_H(u)=\max(1+u,0)$ be the hinge loss and
$\max_{j=1,\ldots,M}\|\phi_j\|_{P_X,\infty}\le L_\phi$ for some
$L_\phi>0$. Then, for every $\beta>0$ the MA aggregate $\hat f_n$ based
on the prior given by (\ref{prior}) satisfies
\begin{eqnarray*}
\Ex_f[\ell_{\Phi_H}(\hat f_n,f)]&\le&\min_{\|\llambdastar\|_1\le
R-2M\tau}\Biggl(
\ell_{\Phi_H}(f_{\llambdastar},f)+\frac{4\beta}{n+1}\sum_{j=1}^M
\log(1+\tau^{-1}|\lambdastar_j|)\Biggr)\\
&&{}+\frac{2(1+RL_\phi
)^2}{\beta}
\mathrm{e}^{(1+RL_\phi)/\beta}+\tilde{\sf R}(M,\tau),
\end{eqnarray*}
where $\tilde{\sf R}(M,\tau)=4\tau
L_\phi\sqrt{M}+\beta(n+1)^{-1}$.\vadjust{\goodbreak}
\end{proposition}

The proof of this proposition is given in the \hyperref[appendix]{Appendix}.

\begin{remark}\label{Remark9}
\begin{enumerate}
%is that parameter $\tau$ of the prior should be chosen differently
%and essentially smaller than it was suggested in (\ref{tau}). One
%easily checks that $\tau=\min(\frac1{n\sqrt{M}},
%$\tilde{\sf R}$ and the inequality $R-2M\tau>R/2$.
%
\item Consider the sparsity scenario, that is, assume that for
some vector $\llambda^*$ having at most $M^*$ nonzero coordinates, the
excess risk $\ell_{\Phi_H}(f_{\llambdastar},f)$ is small and
$\|\llambda^*\|_1\le R/2$. Proposition~\ref{classif2} with the choice
of $\beta=(1+RL_\phi)\sqrt{n/M^*}$ and
$\tau=\min(\frac1{\sqrt{nM}}, \frac{R}{4M})$ leads to the
sparsity oracle inequality
\[
\hspace*{-5pt}\Ex_f[\ell_{\Phi_H}(\hat f_n,f)]\le\!\mathop{\min_{\|\llambdastar\|_1\le R/2}}_{\|\llambdastar\|_0\le M^*
}
\biggl(\!\ell_{\Phi_H}(f_{\llambdastar},f)+\frac{(1+RL_\phi)\sqrt
{M^*}}{\sqrt{n}}
\{C+4\log(1+\tau^{-1}\|\llambdastar\|_1)\}\!\biggr),
\]
where $C>0$ is a constant independent on $M$, $M^*$ and $n$ if $M^*\le
n$. This result is valid for arbitrary $\eta$. It should be noted that
the MA aggregate $\hat f_n$ satisfying this SOI depends on the upper
bound $M^*$ on the sparsity level, which is not always available in
practice. Constructing a classifier independent of $M^*$ and satisfying
the above SOI is an interesting open problem.

\item An important special case is a dictionary composed from a large
number of simple binary classifiers $\phi_j\dvtx\mcX\to\{\pm1\}$. If we
choose $R=1$, all aggregates $\fllambda$ with $\|\llambda\|_1\le R$, as
well as their mixtures, take values in $[-1,1]$, and therefore the
function $Q(z,\fllambda)$ associated with the hinge loss is linear in
$\llambda$. This property has two important consequences. The first one
is that Assumption \ref{assl} holds with $\mcM=0$ and it is no longer necessary
to smooth out the function $L_f(\llambda)$ and to use
Remark~\ref{AssumpL} in the proof Proposition~\ref{classif2}. Thus, the
residual term $\tilde{\sf R}$ is equal to $\beta(n+1)^{-1}$. The second
consequence is computational, related to the Langevin Monte-Carlo
approximation of the MA aggregate briefly described in
Section~\ref{ss-comp} below. Namely, in this case we have strong mixing
properties that are independent of the ambient dimension $M$, due to
the independence of the coordinates of the Langevin diffusion.

\item According to \cite{Lecue07}, if the underlying distribution $P$
satisfies the margin assumption of \cite{tsy04}, then the rate of
aggregation can be substantially improved. It would be interesting to
investigate whether this property extends to the sparsity scenario. It
is likely that one of the randomized procedures of~\cite{a2} used in
conjunction with our sparsity prior can yield an aggregation rate
optimal classifier.
\end{enumerate}
\end{remark}

%s8 ###
%s8 #&#
\section{Discussion}\label{secdisc}

%s8.1 ###
%s8.1 #&#
\subsection{Comparison with other methods of sparse estimation}
In this paper, we have proved sparsity oracle inequalities (SOI) in a
setting, which is important but not much studied in the literature on
sparsity. We considered the i.i.d. random sampling and we measured
the
quality of estimation/prediction by the average loss with respect to
the distribution of $Z=(X,Y)$, namely, our main example was the loss
$\ell(g,f)=\int_\mcZ Q(z,g)\Pf(\mathrm{d}z)$. Most of the literature on sparse
estimation is focused on the high-dimensional linear regression model
with fixed design, so the data are not i.i.d. and the empirical
prediction loss, rather than the average loss is considered. Notable
exceptions are the papers \cite{btw07b,k06,Kol08,kol09,vdg06} where
the framework is similar to ours. Among these, \cite{kol09} focuses on
regression with random design and study the Dantzig selector, while
\cite{btw07b,k06,Kol08,vdg06} analyze the penalized estimators of the
form
\[
\hat\llambda_n= \operatorname{arg} \min_{\llambda\in\Lambda}\Biggl(
\frac1{n}\sum_{i=1}^n Q(Z_i, \flambda) + \mathrm{Pen}(\llambda)\Biggr),
\]
where $\operatorname{Pen}(\llambda)$ is a penalty, which is equal or close to
the $\ell_1$-penalty $r\|\llambda\|_1$ with a suitable regularization
parameter $r>0$. For the penalized estimator $\tilde f_n=f_{\hat
\llambda_n}$, they prove SOI of the form (here we give a ``generic''
simplified version based on \cite{k06}):
%e14 ###
%
%e14 #&#
\begin{equation}\label{koltch}
\ell(\tilde f_n,f)\le\mathop{\min_{\|\llambdastar\|_1\le R}}_{\|\llambdastar\|_0\le M^*
}
\biggl(
3\ell(f_{\llambdastar},f)+\frac{C(1+R^2)M^*}{n\kappa
_{n,M}}{\mathcal
L}_{n,M} \biggr)
\end{equation}
with a probability close to 1, where $C>0$ is a constant independent of
$n$ and $M$, ${\mathcal L}_{n,M}$ is a~factor, which is logarithmic in
$n$ and $M$, and $\kappa_{n,M}$ is minimal sparse eigenvalue appearing
in the conditions on the Gram matrix of the dictionary quoted in the
\hyperref[sec1]{Introduction}. With the same notation, a ``generic'' version of our SOI
for the MA aggregate $\hat f_n$ is the following:
%e15 ###
%
%e15 #&#
\begin{equation}\label{our}
\Ex[\ell(\hat f_n,f)]\le\mathop{\min_{\|\llambdastar\|_1\le R}}_{\|\llambdastar\|_0\le M^*
}
\biggl( \ell(f_{\llambdastar},f)+\frac{C(1+R^2)M^*}{n}{\mathcal L}_{n,M}
\biggr).
\end{equation}
There are two advantages of (\ref{our}) with respect to (\ref{koltch}).
First, (\ref{our}) is a sharp oracle inequality, since the leading
constant is 1, whereas this is not the case in (\ref{koltch}). Second
and most important, (\ref{our}) holds under mild assumptions on the
dictionary, such as the boundedness of the functions $\phi_j$ in some
norm, whereas (\ref{koltch}) requires restrictive assumptions on
minimal sparse eigenvalue $\kappa_{n,M}$ which can be very small and
appears in the denominator. In particular, (\ref{our}) is applicable
when $\kappa_{n,M}=0$. Finally, we note that (\ref{koltch}) is an
oracle inequality ``in probability'' while (\ref{our}) is ``in
expectation.'' Inequalities in expectation can be derived from the
inequalities in probability of the form (\ref{koltch}) obtained in
\cite{btw07b,k06,Kol08,vdg06} only under some additional assumptions.
So, strictly speaking, even more assumptions should be imposed in the
case of (\ref{koltch}) to make possible direct comparison with
(\ref{our}).

In conclusion, we see that the oracle bounds for $\ell_1$-penalized
methods, such as the Lasso or its modifications can be quite inaccurate
as compared to the those that we obtain for the MA aggregate.

The $\ell_0$-penalized methods for models with i.i.d. data are less
studied. To our knowledge, this is done only for regression with random
design \cite{btw04} and for density estimation \cite{klemela}. The
oracle inequalities in those papers are less accurate than the ours
since the leading constant there is greater than 1. Moreover, if we
want to make it closer to 1, the remainder term of the oracle
inequalities explodes.

Furthermore, as mentioned above, our sparsity oracle inequalities are
potentially applicable not only for the MA aggregate, but for any
estimator associated to prior distribution~$\pi$ and satisfying a
PAC-Bayesian bound in expectation as in Theorem \ref{Th1}.

%s8.2 ###
%s8.2 #&#
\subsection{Computational aspects}\label{ss-comp}
If the dimension $M$ is large the computation of the MA aggregate with
the sparsity prior becomes a hard problem. Indeed, its definition
contains integrals over a simplex in $\RR^M$. Nevertheless, accurate
approximations can be realized by a numerically efficient algorithm
based on Langevin Monte Carlo. This algorithm along with the
convergence and simulation studies is discussed in \cite{dt09,dt10}.
Here we only sketch some main ideas underlying the numerical procedure.
For simplicity, we consider the case of linear regression (cf.\
Subsection \ref{Reglin}). The argument is easily extended to other
models discussed in the previous section.

Thus, assume that we have a sample $(\bX_i,Y_i)$, $i=1,\ldots,n$, and a
finite dictionary $\{\phi_j\dvtx\mcX\to\RR\}$ of cardinality $M$. We wish
to compute the expression
\begin{equation}
\tilde\llambda=\frac{\int_{\RR^M} \llambda
\mathrm{e}^{-\beta^{-1}\|\bY-F_\llambda(\bX)\|_2^2}\pi(\mathrm{d}\llambda)}
{\int_{\RR^M}\mathrm{e}^{-\beta^{-1}\|\bY-F_\llambda(\bX)\|_2^2}\pi
(\mathrm{d}\llambda)},
\end{equation}
where $F_\llambda(\bX)=(\fllambda(\bX_1),\ldots,\fllambda(\bX
_n))^\T$
and $\fllambda=\sum_{j=1}^M\lambda_j\phi_j$. A slight modification of
the sparsity prior consists in replacing $\pi$ defined in (\ref{prior})
by
%e16 ###
%
%e16 #&#
\begin{equation}\label{pi-1}
\tilde\pi(\mathrm{d}\llambda)\propto\Biggl(\prod_{j=1}^M
\frac{\mathrm{e}^{-\opi(\alpha\lambda_j)}}
{(\tau^2+\lambda_j^2)^2}\Biggr)\1(\|\llambda\|_1\le R)\,\mathrm{d}\llambda,
\end{equation}
where $\alpha$ is a small parameter and $\opi\dvtx\RR\to\RR$ is the Huber
function: $\opi(t)=t^2\1(|t|\le1)+(2|t|-1)\1(|t|>1)$. Introducing the
product of $e^{-\opi(\alpha\lambda_j)}$ in the definition of the prior
does not affect its capacity to capture sparse objects, in the sense
that the MA aggregate based on the prior~(\ref{pi-1}) can be shown to
satisfy a SOI which is quite similar to that of
Theorem~\ref{ThmSparsity} (cf. \cite{dt09,dt10} where the regression
model with fixed design is treated). On the other hand, this
modification of the sparsity prior makes it possible to rigorously
prove the geometric ergodicity of the Langevin diffusion defined below.

Note that we can equivalently write $\tilde\llambda$ in the form
\begin{equation}\label{lll}
\tilde\llambda=\frac{\int_{\RR^M}\llambda\1(\|\llambda\|_1\le R)
p_V(\llambda)\,\mathrm{d}\llambda} {\int_{\RR^M}\1(\|\llambda\|_1\le R)
p_V(\llambda)\,\mathrm{d}\llambda},
\end{equation}
where $p_V(\llambda)\propto e^{V(\llambda)}$ with
\begin{equation}\label{V}
V(\llambda)=-\beta^{-1}\|\bY-F_\llambda(\bX)\|_2^2-\sum_{j=1}^M
2\{\log(\tau^2+\lambda_j^2)+ \opi(\alpha\lambda_j)\}.
\end{equation}
Consider now the Langevin stochastic differential equation (SDE)
\[%\label{Langevin}
\mathrm{d}\bL_t=\nabla V(\bL_t)\,\mathrm{d}t+\sqrt{2}\,\mathrm{d}\bW_t,\qquad \bL
_0=0, t\ge0,
\]
where $\bW$ stands for an $M$-dimensional Brownian motion. For our
choice of the potential~$V$, this SDE has a unique strong solution. It
can be also shown (cf. \cite{dt09,dt10}) that this choice of $V$
guarantees the geometric ergodicity of the solution, which implies that
its stationary distribution has the density $p_V(\llambda)\propto
\mathrm{e}^{V(\llambda)}$, $\llambda\in\RR^M$. This and (\ref{lll}) suggest the
Langevin Monte Carlo procedure of computation of $ \tilde\llambda$.
Indeed, consider the time averages
\[
\bar\bL_T=\frac1T\int_0^T \bL_t\1(\|\bL_t\|_1\le R)\,\mathrm{d}t,\qquad
S_T=\frac1T\int_0^T \1(\|\bL_t\|_1\le R)\,\mathrm{d}t,\qquad  T\ge0.
\]
According to the above remarks, the ratio of these average values
converges, as $T\to\infty$, to the vector $\tilde\llambda$ that we want
to compute. Note that $\bar\bL_T$ and $S_T$ are one-dimensional
integrals over a finite interval and, therefore, are simpler objects
than $\tilde\llambda$, which is an integral in $M$ dimensions. Still,
one cannot compute $\bar\bL_T$ directly, and some discretization is
needed. A standard way of doing it is to approximate $\bar\bL_T$ and
$S_T$ by the sums
\[
\bar\bL_{T,h}^E=\frac1{[T/h]}\sum_{k=0}^{[T/h]-1}
\bL_k^E\1(\|\bL_k^E\|_1\le R),\qquad
S_{T,h}^E=\frac1{[T/h]}\sum_{k=0}^{[T/h]-1} \1(\|\bL_k^E\|_1\le R),
\]
where $\{\bL_k^E\}$ is the Markov chain defined by the Euler scheme
\[
\bL_{k+1}^E=\bL_{k}^E+h\nabla V(\bL_k^E)+\sqrt{2h}\bW_k,\qquad
\bL^E_0=0,  k=0,1,\ldots,[T/h]-1.
\]
Here $\bW_1$, $\bW_2,\ldots$ are i.i.d. standard Gaussian random
vectors in $\RR^M$, $h>0$ is a step of discretization, and $[x]$ stands
for the integer part of $x\in\RR$. It can be shown that $
\bar\bL_{T,h}^E$ is an accurate approximation of $\bar\bL_T$ for small
$h$. We refer to \cite{dt09,dt10} for further details. The
computational complexity is polynomial in $M$ and $n$. Simulation
results in \cite{dt09,dt10}, as well as the experiments on image
denoising \cite{Jo}, show the fast convergence of the algorithm; it can
be easily realized in dimensions $M$ up to several thousands. They also
demonstrate nice performance of the exponentially weighted aggregate as
compared with the Lasso and other related methods of prediction under
the sparsity scenario.

\begin{appendix}\label{appendix}
\section*{Appendix}\label{sec8}
\setcounter{subsection}{0}
\renewcommand{\theequation}{\arabic{equation}}

%s9.1 ###
%s9.1 #&#
\subsection{\texorpdfstring{Proof of Theorem \protect\ref{Th1}}{Proof of Theorem 1}}

First, note that
without loss of generality we can set $\beta=1$. If this is not the
case, it suffices to replace $Q$ and $\ell$ by $\tilde Q=\frac1\beta Q$
and $\tilde\ell=\frac1\beta\ell$, respectively. By Assumption \ref{assq1},
%e17 ###
%
%e17 #&#
\begin{equation}\label{eq2}
\Ex_f[\ell(\hat f_n,f)]=\int_\mcZ\Ex_f[Q(z,\hat
f_n)]\Pf(\mathrm{d}z)-\Delta(f).
\end{equation}
In the last display we have used Fubini's theorem to interchange the
integral and the expectation; this is possible since the integrand is
bounded from below. To get the desired result, one needs now to bound
the first term on the RHS of (\ref{eq2}), which we rewrite as follows
%e18 ###
%
%e18 #&#
\begin{equation}\label{eq3}
\int_\mcZ\Ex_f[Q(z,\hat
f_n)]\Pf(\mathrm{d}z)=-\int_\mcZ\Ex_f[\log(\exp\{
-Q(z,\hat
f_n)\})]\Pf(\mathrm{d}z).
\end{equation}
Recall now that $\hat f_n$ is defined as the average of the functions
$\flambda$ w.r.t. the probability measure $\hat\mu_n$. If we knew that
the mapping $g\mapsto\exp\{-Q(z,g)\}$ is concave on the convex
hull of $\mcF_\Lambda$, we could apply Jensen's inequality to get
\[
\exp\{-Q(z,\hat f_n)\} \ge\int_\Lambda
\exp\{-Q(z,\flambda)\} \hat\mu_n(\mathrm{d}\llambda).
\]
As we see below, this would allow us to get inequality (\ref{oracle})
by a simple application of the convex duality argument. Unfortunately,
the above mentioned concavity property is rather exceptional and
therefore the quantity
\[
S_1(z,\bZ)=\log\biggl(\int_\Lambda\exp\{-Q(z,\flambda)\}
\hat\mu_n(\mathrm{d}\llambda)\biggr)-\log(\exp\{-Q(z,\hat f_n)
\})
\]
is not necessarily a.s. negative. However, we may write
\begin{equation}\label{eq4}
\int_\mcZ\Ex_f\bigl[\log\bigl(\mathrm{e}^{-Q(z,\hat f_n)}\bigr)\bigr]\Pf(\mathrm{d}z)=
\int_\mcZ\Ex_f[S_0(z,\bZ)-S_1(z,\bZ)]\Pf(\mathrm{d}z),
\end{equation}
where
\[
S_0(z,\bZ)=\log\biggl(\int_\Lambda\exp\{-Q(z,\flambda)\}
\hat\mu_n(\mathrm{d}\llambda)\biggr).
\]
By the concavity of the logarithm,
\[
S_0(z,\bZ) \ge\frac1{n+1}\sum_{m=0}^n \log\biggl(\int_\Lambda
\mathrm{e}^{-Q(z,\flambda)} \hat\theta_{m,\llambda}\pi(\mathrm{d}\llambda)\biggr).
\]
Replacing $\hat\theta_{m,\llambda}$ by its explicit expression and
taking the integral of both sides of the last display, we get on the
RHS a telescoping sum. This leads to the inequality
\[
\int_\mcZ\Ex_f[S_0(z,\bZ)]\Pf(\mathrm{d}z) \ge\frac1{n+1}
\int_{\mcZ^{n+1}}\log\biggl(\int_\Lambda \mathrm{e}^{-\sum_{i=1}^{n+1}
Q(z_i,\flambda)} \pi(\mathrm{d}\llambda)\biggr) \Pf^{(n+1)}(\mathrm{d}{\mathbf z}).
\]
By a convex duality argument (cf., e.g., \cite{dz}, page 264, or
\cite{catbook01}, page 160), we get
\begin{eqnarray*}
\log\biggl(\int_\Lambda \mathrm{e}^{-\sum_{i=1}^{n+1} Q(z_i,\flambda)}
\pi(\mathrm{d}\llambda)\biggr) \ge-\sum_{i=1}^{n+1} \int_\Lambda
Q(z_i,\flambda) p(\mathrm{d}\llambda)-\KL(p,\pi),
\end{eqnarray*}
for every $p\in\mcPLambda$. Therefore, integrating w.r.t.
$z_1,\ldots,z_{n+1}$ and using the symmetry, we get
\begin{eqnarray*}
\int_\mcZ\Ex_f[S_0(z,\bZ)]\Pf(\mathrm{d}z)
&\ge&-\int_\mcZ\int_\Lambda Q(z,\flambda) p(\mathrm{d}\llambda) \Pf
(\mathrm{d}z)-\frac{\KL(p,\pi)}{n+1}\\
&=&-\int_\Lambda\ell(\flambda,f) p(\mathrm{d}\llambda)
-\Delta(f)-\frac{\KL(p,\pi)}{n+1}.
\end{eqnarray*}
This and equations (\ref{eq2})--(\ref{eq4}) imply
%e19 ###
%
%e19 #&#
\begin{equation}\label{eq5}
\Ex_f[\ell(\hat f_n,f)]\le\int_\Lambda\ell(\flambda,f)
p(\mathrm{d}\llambda)
+ \frac{\KL(p,\pi)}{n+1}+\int_\mcZ\Ex_f[S_1(z,\bZ)]\Pf(\mathrm{d}z).
\end{equation}
Let us show that the last term on the RHS of (\ref{eq5}) is
nonpositive. Rewrite $ S_1(z,\bZ)$ in the form
\[
S_1(z,\bZ)=\log\int_\Lambda\exp\bigl(-\{Q(z,\flambda
)-Q(z,\hat
f_n)\}\bigr)\hat\mu_n(\mathrm{d}\llambda).
\]
By the Fubini theorem, the concavity of the logarithm and Assumption
\ref{assq2}, we get
\[
\int_\mcZ\Ex_f[S_1(z,\bZ)]\Pf(\mathrm{d}z)\le\Ex_f\biggl[\log\int
_\Lambda
\Psi_1(\flambda,\hat f_n) \hat\mu_n(\mathrm{d}\llambda)\biggr]
\]
(recall that we set $\beta=1$). The concavity of the map
$g\mapsto\Psi_1(g,\hat f_n)$ and Jensen's inequality yield
\[
\int_\Lambda\Psi_1(\flambda,\hat f_n)\hat\mu_n(\mathrm{d}\llambda)\le
\Psi_1\biggl(\int_\Lambda\flambda\hat\mu_n(\mathrm{d}\llambda),\hat
f_n\biggr)=\Psi_\beta(\hat f_n,\hat f_n)=1,
\]
and the desired result follows.

%s9.2 ###
%s9.2 #&#
\subsection{Some lemmas}
We now give some technical results needed in
the proofs.

%l1 #&#
\begin{lemma}\label{lem0}
For every $M\in\mathbb N$ and every $s>M$, the following inequality
holds:
\[
\frac1{(\pi/2)^M}\int_{\{u:\|u\|_1>s\}}
\prod_{j=1}^M\frac{\mathrm{d}u_j}{(1+u_j^2)^2}\le\frac{M}{(s-M)^2}.
\]
\end{lemma}
\begin{pf}
Let $U_1,\ldots,U_M$ be i.i.d. random variables drawn from the scaled
Student $t(3)$ distribution having as density the function $u\mapsto
2/[\pi(1+u^2)^2]$. One easily checks that $\Ex[U_1^2]=1$.
Furthermore, with this notation, we have
\[
\frac1{(\pi/2)^M}\int_{\{u:\|u\|_1>s\}}
\prod_{j=1}^M\frac{\mathrm{d}u_j}{(1+u_j^2)^2}=\Pb\Biggl(\sum_{j=1}^M
|U_j|\ge
s\Biggr).
\]
In view of Chebyshev's inequality, the last probability can be bounded
as follows:
\[
\Pb\Biggl(\sum_{j=1}^M |U_j|\ge
s\Biggr)\le\frac{M\mathbf{E}[U_1^2]}{(s-M\Ex[|U_1|])^2}\le
\frac{M}{(s-M)^2}
\]
and the desired inequality follows.
\end{pf}

%l2 #&#
\begin{lemma}\label{lem1}
Let the assumptions of Theorem~\ref{ThmSparsity} be satisfied and let
$p_0$ be the probability measure defined by (\ref{eq6}). If $M\ge2$,
then $\int_\Lambda(\lambda_1-\lambda_1^*)^2p_0(\mathrm{d}\llambda)\le4\tau^2$.
\end{lemma}

\begin{pf}
Using the change of variables $u=(\llambda-\llambdastar)/\tau$, we
write
\[
\int_\Lambda(\lambda_1-\lambda_1^*)^2p_0(\mathrm{d}\llambda)=
C_{M}\tau^2\int_{B_1(2M)} u_1^2\Biggl(\prod_{j=1}^M
(1+u_j^2)^{-2}\Biggr)\,\mathrm{d}u
\]
with
\begin{equation}\label{CM}
C_M=\Biggl(\int_{B_1(2M)}\Biggl(\prod_{j=1}^M
(1+u_j^2)^{-2}\Biggr)\,\mathrm{d}u\Biggr)^{-1},
\end{equation}
where $u_j$ are the components of $u$. Extending the integration from
$B_1(2M)$ to $\RR^M$ and using the inequality $\int_\RR
u_1^2(1+u_1^2)^{-2}\,\mathrm{d}u_1\le\pi$, we get
\[
\int_\Lambda(\lambda_1-\lambda_1^*)^2p_0(\mathrm{d}\llambda)\le C_{M}\tau
^2\pi
\biggl(\int_{\RR} (1+t^2)^{-2}\,\mathrm{d}t\biggr)^{M-1}= 2C_{M}\tau^2(\pi/2)^M,
\]
where we used that the primitive of the function $(1+x^2)^{-2}$ is
$\frac12\arctan(x)+\frac{x}{2(1+x^2)}$. To bound $C_M$, we apply
Lemma~\ref{lem0} which yields
\begin{equation}\label{inneq1}
C_M\le(2/\pi)^M(1-1/M)^{-1} \le2(2/\pi)^M,
\end{equation}
for $M\ge2$. Combining these estimates, we get $\int_\Lambda
(\lambda_1-\lambda_1^*)^2p_0(\mathrm{d}\llambda)\le4\tau^2$ and the desired
inequality follows.
\end{pf}

%l3 #&#
\begin{lemma}\label{lem2}
Let the assumptions of Theorem~\ref{ThmSparsity} be satisfied and let
$p_0$ be the probability measure defined by (\ref{eq6}). Then
$\KL(p_0,\pi)\le4\sum_{j=1}^M \log(1+|\lj^*|/\tau)+1$.
\end{lemma}

\begin{pf}
The definition of $\pi$, $p_0$ and of the Kullback--Leibler divergence
imply that
\begin{eqnarray}\label{eq7}
\KL(p_0,\pi)&=&\int_{B_1(2M\tau)} \log\Biggl\{\tau^{3M}C_M
C_{\tau,R}\prod_{j=1}^M
\frac{(\tau^2+\lj^2)^2}{(\tau^2+(\lj-\lj^*)^2)^2}\Biggr\}
p_0(\mathrm{d}\llambda)\nonumber\\ [-8pt]\\ [-8pt]
&=&\log(\tau^{3M} C_M C_{\tau,R})+2\sum_{j=1}^M\int_{B_1(2M\tau)}
\log\biggl\{ \frac{\tau^2+\lj^2}{\tau^2+(\lj-\lj^*)^2}\biggr\}\nonumber
p_0(\mathrm{d}\llambda).
\end{eqnarray}
We now successively evaluate the terms on the RHS of (\ref{eq7}).
First, in view of (\ref{prior}), we have
\[
C_{\tau,R}=\tau^{-3M}\int_{B_1(R/\tau)} \prod_{j=1}^M
\frac{1}{(1+u_j^2)^2}\,\mathrm{d}u_j \le\tau^{-3M}\biggl(\int_\RR
(1+u_j^2)^{-2}\,\mathrm{d}u_j\biggr)^M=\tau^{-3M}(\pi/2)^M.
\]
This and (\ref{inneq1}) imply $\log(C_M C_{\tau,R})\le\log2\le 1$.

To evaluate the second term on the RHS of (\ref{eq7}), we use that
\begin{eqnarray*}
\frac{\tau^2+\lambda^2_j}{\tau^2+(\lj-\lj^*)^2}&=&1+\frac{2\tau
(\lj-\lj^*)}{\tau^2+(\lj-\lj^*)^2}(\lj^*/\tau)+\frac{{\lj
^*}^2}{\tau^2+(\lj-\lj^*)^2}\\
&\le&1+|\lj^*/\tau|+(\lj^*/\tau)^2\le(1+|\lj^*/\tau|)^2.
\end{eqnarray*}
This entails that the second term on the RHS of (\ref{eq7}) is bounded
from above by $\sum_{j=1}^M 2\log(1+|\lj^*|/\tau)$. Combining these
inequalities, we get the lemma.\vspace*{-3pt}
\end{pf}

%s9.3 ###
%s9.3 #&#
\subsection{\texorpdfstring{Proof of Theorem \protect\ref{ThmSparsity}}{Proof of Theorem 2}}\vspace*{-3pt}
\label{sec93}
In view of inequality (\ref{oracle}), we have
\[
\Ex_f[\ell(\hat f_n,f)]\le\int_\Lambda
\ell(\fllambda,f)p(\mathrm{d}\llambda)+\frac{\beta\KL(p,\pi)}{n+1},
\]
for every probability measure $p$. We choose here $p=p_0$ where $p_0$
has the following Lebesgue density:
%e20 ###
%
%e20 #&#
\begin{equation}\label{eq6}
\frac{\mathrm{d}p_0}{\mathrm{d}\llambda}(\llambda)\propto\frac{\mathrm{d}\pi}{\mathrm{d}\llambda}
(\llambda-\llambdastar)\1_{B_1(2M\tau)}(\llambda-\llambdastar).
\end{equation}
Here the sign $\propto$ indicates the proportionality of two functions.
Since $\|\llambdastar\|_1\le R-2M\tau$, the condition
$\llambda-\llambdastar\in B_1(2M\tau)$ implies that $\llambda\in
B_1(R)$ and, therefore, $p_0$ is absolutely continuous w.r.t. the
sparsity prior $\pi$. By Taylor's formula and Assumption \ref{assl}, we have
\[
\ell(\fllambda,f)=L_f(\llambda)\le L_f(\llambdastar)+\nabla
L_f(\llambdastar)^\top(\llambda-\llambdastar)+
(\llambda-\llambdastar)^\top  \mcM(\llambda-\llambdastar)\qquad
\forall\llambda\in\Lambda_0.
\]
Integrating both sides of this inequality w.r.t. $p_0$ and using the
fact that the density of~$\pi_0$ is symmetric about $\llambdastar$ and
invariant under permutation of the components we find
\begin{equation}
\int_\Lambda\ell(\fllambda,f)p_0(d\llambda)\le
L_f(\llambdastar)+\tr(\mcM)\int_\Lambda
(\lambda_1-\lambda_1^*)^2p_0(\mathrm{d}\llambda).
\end{equation}
Combining this inequality with those stated in Lemmas \ref{lem1} and
\ref{lem2}, we get the desired result.\vspace*{-3pt}

%s9.4 ###
%s9.4 #&#
\subsection{\texorpdfstring{Proof of Proposition \protect\ref{propdensity}}{Proof of Proposition 4}}\vspace*{-3pt}
Note that Assumption \ref{assq1} obviously holds and Assumption \ref{assl} is fulfilled
with $\mcM$ being the Gram matrix. The diagonal entries of $\mcM$ are
equal to one since $\|\phi_j\|_{\mu,2}=1$, and therefore we have
$\tr(\mcM)=M$.\vadjust{\goodbreak}

It remains to check Assumption \ref{assq2} in order to apply
Theorem~\ref{ThmSparsity}. Introduce the function
\begin{eqnarray*}
\Xi(t)&=&\exp\bigl(-\beta^{-1}\bigl\{
Q\bigl(X_1,g_0+t(g_1-g_0)\bigr)-Q(X_1,\tilde g)\bigr\}\bigr)\\
&=&\exp\bigl[-\beta^{-1}\bigl\{\|g_t\|_{\mu,2}^2-\|\tilde g\|_{\mu,2}^2
+2\bigl(\tilde g(X_1)-g_t(X_1)\bigr)\bigr\}\bigr],\qquad   t\in[0,1],
\end{eqnarray*}
where $g_0$, $g_1$ and $\tilde g$ are functions from the convex set
$\mcFL$, and $g_t=g_0+t(g_1-g_0)\in\mcF$. It is not hard to see that
Assumption \ref{assq2} follows from the fact that the mapping $t\mapsto
\Ex_f[\Xi(t)]$ is concave for any triplet $g_0,g_1, \tilde g \in
\mcFL$. Let us prove now this concavity property. Since the functions
$g_0,g_1, \tilde g$ are uniformly bounded we get that $\Xi(\cdot)$ is
twice continuously differentiable and the differentiation inside the
expectation $\Ex_f[\Xi(t)]$ is legitimate. Therefore,
\begin{eqnarray*}
\frac{\mathrm{d}}{\mathrm{d} t} \Ex_f[\Xi(t)]&=&-2\beta^{-1}\Ex_f\bigl[
\bigl(\langle g_t,h\rangle-h(X_1)\bigr)\Xi(t)\bigr],\\
\frac{\mathrm{d}^2}{\mathrm{d} t^2} \Ex_f[\Xi(t)]&=&-2\beta^{-2}\Ex_f\bigl[\bigl(
\beta\|h\|_2^2-2\{\langle g_t,h\rangle-h(X_1)\}^2 \bigr)
\Xi(t)\bigr],
\end{eqnarray*}
where $h=g_1-g_0$, and
\[
\frac{\beta^2}{2} \frac{\mathrm{d}^2}{\mathrm{d} t^2}
\Ex_f[\Xi(t)]\le
-(\beta\|h\|_2^2-2\langle
g_t,h\rangle^2)\Ex_f[\Xi(t)]+2\Ex_f[\{
h(X_1)^2-2\langle
g_t,h\rangle h(X_1)\} \Xi(t)].
\]
This leads to
\[
\Xi(t)\le\exp[-\beta^{-1}\{\|g_t\|_{\mu,2}^2-\|\tilde
g\|_{\mu,2}^2\}+4 R L/\beta]:=\Xi_1(t)
\]
and
\[
\Ex_f[\Xi(t)]\ge
\exp\Bigl[-\beta^{-1}\Bigl\{\|g_t\|_{\mu,2}^2-\|\tilde
g\|_{\mu,2}^2+4\max_{\mcFL}\Ex_f[|g(X_1)|]\Bigr\}\Bigr]
=\Xi_1(t)\mathrm{e}^{-4R(L+\sqrt{L})/\beta}.
\]
Combining these estimates with inequalities
\[
\Ex[h(X_1)^2]\le L\|h\|_2^2,\qquad |\langle g_t,h\rangle| \le
\|g_t\|_2\|h\|_2\le R\|h\|_2,\qquad \Ex[| \langle g_t,h\rangle
h(X_1)|]\le RL\|h\|_2^2,
\]
we get
\[
\frac{\beta^2}{2} \frac{\mathrm{d}^2}{\mathrm{d}t^2} \Ex_f[\Xi(t)]\le
-\|h\|_2^2\Xi_1(t)\bigl((\beta-2R^2)\mathrm{e}^{-4R(L+\sqrt{L})/\beta}-2L-4R
L\bigr)\le0,
\]
whenever $(\beta-2R^2)\mathrm{e}^{-4R(L+\sqrt{L})/\beta}\ge2L+4R L$. This
proves the concavity of $t\mapsto\Ex_f[\Xi(t)]$, and thus the
proposition.

%s9.5 ###
%s9.5 #&#
\subsection{\texorpdfstring{Proof of Proposition \protect\ref{classif2}}{Proof of Proposition 6}}
In view of (\ref{eq5}), for any prior $\pi$ and any $\beta>0$ the MA
aggregate $\hat f_n$ satisfies the inequality
%e21 ###
%
%e21 #&#
\begin{equation}
\hspace*{-5pt}\Ex_f[\ell_\Phi(\hat f_n,f)]\le
\inf_{p\in\mcPLambda}\biggl(\int_{\Lambda}
\ell_\Phi(\flambda,f)p(\mathrm{d}\llambda)+\frac{\beta\KL(p,\pi
)}{n+1}\biggr)+\beta\int_\mcZ\Ex_f[S_1(z,\bZ)]\Pf(\mathrm{d}z)
\end{equation}
with $ S_1(z,\bZ)$ defined by
$S_1(z,\bZ)=\log\int_\llambda\exp(-\beta^{-1}\{
Q(z,\flambda)-Q(z,\hat
f_n)\})\hat\mu_n(\mathrm{d}\llambda)$. Let us introduce the function
$\psi_\llambda(t)=\exp(-t\{Q(z,\flambda)-Q(z,\hat
f_n)\})$. This function is infinitely differentiable, equals
one at the origin and we have $S_1(z,\bZ)=\log\int_\Lambda
\psi_\llambda(\beta^{-1})\hat\mu_n(\mathrm{d}\llambda)$. Using the Taylor
formula, we get
\[
\psi_\llambda(t)\le1+t\psi_\llambda'(0)+\frac{t^2}2
\bigl(Q(z,\flambda)-Q(z,\hat f_n)\bigr)^2\mathrm{e}^{tQ(z,\hat f_n)}\qquad
\forall
t\ge0.
\]
Furthermore, since the hinge loss is convex, the Jensen inequality
yields $\int_\Lambda\psi_\llambda'(0)\hat\mu_n(\mathrm{d}\llambda)\le0$.
Replacing $t$ by $\beta^{-1}$ and using that $Q(z,\hat f_n)\le
1+RL_\phi$, we get the inequalities
\begin{eqnarray*}
S_1(z,\bZ)&=&\log\int_\Lambda
\psi_\llambda(\beta^{-1})\hat\mu_n(\mathrm{d}\llambda)\le
\log\biggl(1+\frac{\mathrm{e}^{(1+RL_\phi)/\beta}}{2\beta^2}\int_\Lambda
\bigl(Q(z,\flambda)-Q(z,\hat f_n)\bigr)^2\hat\mu_n(\mathrm{d}\llambda)\biggr)\\
&\le&\frac{\mathrm{e}^{(1+RL_\phi)/\beta}}{2\beta^2}\int_\Lambda
\bigl(Q(z,\flambda)-Q(z,\hat f_n)\bigr)^2\hat\mu_n(\mathrm{d}\llambda)\le
\frac{2\mathrm{e}^{(1+RL_\phi)/\beta}}{\beta^2}(1+RL_\phi)^2.
\end{eqnarray*}
Thus, we obtain
%e22 ###
%
%e22 #&#
\begin{equation}\label{eq19}
\hspace*{-10pt}\Ex_f[\ell_\Phi(\hat f_n,f)]\le
\inf_{p\in\mcPLambda}\biggl(\int_{\Lambda}
\ell_\Phi(\flambda,f)p(\mathrm{d}\llambda)+\frac{\beta\KL(p,\pi
)}{n+1}\biggr)+
\frac{2(1+RL_\phi)^2\mathrm{e}^{(1+RL_\phi)/\beta}}{\beta},
\end{equation}
which is valid for any prior $\pi$. Note that the term with the infimum
in (\ref{eq19}) coincides with the right hand side of the oracle
inequality of Theorem~\ref{Th1}. Therefore, when the sparsity prior is
used, this term can be bounded from above using Remark~\ref{AssumpL}
with $\bar
L_f(\llambda)=\int_\mcX|\eta(x)|K_\eps(\fllambda(x)-f(x)
)
P_X(\mathrm{d}x)$. Since also $|\eta(x)|\le1$, we get
\begin{eqnarray*}
\Ex_f[\ell(\hat f_n,f)]&\le&\min_{\|\llambdastar\|_1\le R-2M\tau
}\Biggl(
\ell(f_{\llambdastar},f)+\frac{2\beta}{n+1}\Biggl\{\alpha\|
\llambdastar\|_1+\sum_{j=1}^M
\log(1+\tau^{-1}|\lambdastar_j|)\Biggr\}\Biggr)\\
&&{}+\eps+4\tau^2\tr
(\bar\mcM_\eps)+\frac{2(1+RL_\phi)^2\mathrm{e}^{(1+RL_\phi)/\beta}}{\beta},
\end{eqnarray*}
where the entries of the matrix $\bar\mcM_\eps$ are $\eps^{-1}\int
_\mcX
|\eta(x)|\phi_j(x)\phi_{j'}(x)P_X(\mathrm{d}x)$ with $i,j=1,\ldots,M$. Thus,
$\tr(\bar\mcM_\eps)\le L_\phi^2 M\eps^{-1}$, and we get the
result of
the proposition by minimizing the right-hand side of the last display
with respect to $\eps>0$.
\end{appendix}

\section*{Acknowledgements}
The authors acknowledge the financial support by ANR under grant PARCIMONIE.

% imsref loaded by svajune.rapalyte, 2011-10-11 13:12:31
%

\printhistory

\end{document}